\documentclass[11pt, twoside, leqno]{article}

\usepackage{amssymb}
\usepackage{amsmath}
\usepackage{amsthm}
\usepackage{color}
\usepackage{mathrsfs}

\usepackage{indentfirst}

\usepackage{txfonts}
\usepackage{bbm}

\allowdisplaybreaks

		\def\XXint#1#2#3{{\setbox0=\hbox{$#1{#2#3}{\int}$}	
\vcenter{\hbox{$#2#3$}}\kern-.5\wd0}}

\def\red{\color{red}}

\def\rr{{\mathbb R}}
\def\rn{{\mathbb{R}^n}}

\def\zz{{\mathbb Z}}

\def\nn{{\mathbb N}}

\def\cm{{\mathcal M}}

\def\cp{{\mathcal P}}

\def\cs{{\mathcal S}}

\def\lf{\left}
\def\r{\right}

\def\noz{\nonumber}

\def\c{\mathop\mathrm{c}}
\def\loc{{\mathop\mathrm{\,loc\,}}}
\def\supp{\mathop\mathrm{\,supp\,}}

\def\XXint#1#2#3{{\setbox0=\hbox{$#1{#2#3}{\int}$ }
\vcenter{\hbox{$#2#3$ }}\kern-.6\wd0}}

\DeclareMathOperator{\esssup}{ess\,sup}

\def\({\left(}
\def \){ \right)}

 \def\supp{\operatorname{supp}}

\newtheorem{theorem}{Theorem}[section]
\newtheorem{lemma}[theorem]{Lemma}
\newtheorem{corollary}[theorem]{Corollary}

\theoremstyle{definition}
\newtheorem{remark}[theorem]{Remark}
\newtheorem{definition}[theorem]{Definition}
\renewcommand{\appendix}{\par
   \setcounter{section}{0}%
   \setcounter{subsection}{0}%
   \setcounter{subsubsection}{0}%
   \gdef\thesection{\@Alph\c@section}%
   \gdef\thesubsection{\@Alph\c@section.\@arabic\c@subsection}%
   \gdef\theHsection{\@Alph\c@section.}%
   \gdef\theHsubsection{\@Alph\c@section.\@arabic\c@subsection}%
   \csname appendixmore\endcsname
 }

\numberwithin{equation}{section}

\begin{document}

\title{\bf\Large Real-Variable Characterizations and Their
Applications
of Matrix-Weighted Triebel--Lizorkin Spaces
\footnotetext{\hspace{-0.35cm}  2020 \emph{Mathematics Subject
Classification}. Primary 46E35; Secondary 42B25, 42B15, 42B35.																																									 \endgraf
\emph{Key words and phrases}. matrix weight, Triebel--Lizorkin space,
Peetre maximal function, Littlewood--Paley function, Fourier multiplier.
\endgraf This project is partially supported
by the National Key Research and Development Program of China
(Grant No. 2020YFA0712900)
and the National Natural Science Foundation of China
(Grant Nos. 11971058 and  12071197).}}
\author{Qi Wang, Dachun Yang\footnote{Corresponding author,
E-mail: \texttt{dcyang@bnu.edu.cn}/{\red July 18, 2022}/Final version.}
\ and Yangyang Zhang}
\date{}
\maketitle

\vspace{-0.7cm}

\begin{center}
\begin{minipage}{13cm}
{\small {\bf Abstract}\quad
Let $\alpha\in\mathbb R$,  $q\in(0,\infty]$,
$p\in(0,\infty)$, and $W$ be an $A_p(\mathbb{R}^n,\mathbb{C}^m)$-matrix weight.
In this article, the authors characterize the matrix-weighted Triebel--Lizorkin
space $\dot{F}_{p}^{\alpha,q}(W)$ via the Peetre maximal function, the Lusin area
function, and the Littlewood--Paley $g_{\lambda}^{*}$-function.
As applications, the authors establish the boundedness of Fourier multipliers on
matrix-weighted Triebel--Lizorkin spaces under the
generalized H\"ormander condition.
The main novelty of these results exists in that their proofs
need to fully use both the doubling property of
matrix weights and the reducing operator
associated to matrix weights, which are essentially different
from those proofs of the corresponding cases of
classical Triebel--Lizorkin spaces that strongly depend
on the Fefferman--Stein vector-valued maximal inequality on Lebesgue spaces.
}
\end{minipage}
\end{center}

\vspace{0.2cm}



\section{Introduction\label{s1}}

Lizorkin \cite{l72,l74} and Triebel \cite{T73} independently
started to investigate Triebel--Lizorkin spaces $F_{p}^{\alpha,q}(\rn)$
from 1970s. Furthermore, we mention the contributions \cite{P3,P2,P4} of Peetre
who extended the range of the admissible parameters $p$ and $q$ to values
less than one. We refer the reader to \cite{T1,T2,T3,Sa2,S3} for more
studies of these function spaces and their history.

On the other hand, the real-variable theory of both function spaces
and the boundedness of operators related to matrix weights on $\rn$ has
received increasing interest in recent years. In 1997, to solve some
significative problems related to the multivariate random stationary
process and the Toeplitz operators (see, for instance, \cite{tv}), Treil
and Volberg \cite{t} introduced the Muckenhoupt $A_2(\rn,\mathbb{C}^m)$-matrix
weights and  generalized the Hunt--Muckenhoupt--Wheeden theorem to the
vector-valued case, while Nazarov and  Treil \cite{nt} introduced Muckenhoupt
$A_p(\rn,\mathbb{C}^m)$-matrix weights for any $p\in(1,\infty)$ (see also Definition \ref{Ap} below for its
definition), and obtained the boundedness of the Hilbert
transform on the matrix-weighted Lebesgue space $L^p(W)$, which was
proved again by  Volberg \cite{v} via a new
approach involving the classical  Littlewood--Paley theory.
In 2016, Cruz-Uribe et al. \cite{cmr16}
applied the theory of $ A_p $ matrix weights on $ \mathbb{R}^{n} $ to study
degenerate Sobolev spaces.
See also, for instance, \cite{ci22, cim18,cimpr,dly21} for more studies
on matrix-weighted function spaces and their applications.
Later, Frazier and Roudenko \cite{FR} introduced the matrix-weighted homogeneous
Triebel--Lizorkin space $\dot{F}^{\alpha,q}_{p}(W)$ via the discrete
Littlewood--Paley $g$-function with $\alpha\in\rr$, $p\in(0,\infty)$,
and $q\in(0,\infty]$ (see also Definition \ref{FBW}
below for its definition). For any given $p\in(1,\infty)$, Frazier and Roudenko \cite{FR}
proved that $L^p(W)=\dot{F}^{0,2}_{p}(W)$ and, for any $k\in\nn$,
$F^{k,2}_{p}(W)$ coincides with the matrix-weighted Sobolev space $L^p_{k}(W)$.
Frazier and Roudenko \cite{FR}
also showed that a vector-valued function $\vec{f}$ belongs to
$\dot{F}^{\alpha,q}_{p}(W)$
if and only if its $\varphi-$transform coefficients belong to the
sequence space
$\dot{f}^{\alpha,q}_{p}(W)$. As an application of the above results,
Frazier and Roudenko
\cite{FR} obtained the boundedness of  Calder\'on--Zygmund operators on
$\dot{F}^{\alpha,q}_{p}(W)$. However,
\emph{no other real-variable characterizations} of these
Triebel--Lizorkin spaces are known so far.
The main purpose of this article is try to fill this gap.

Let $\alpha\in\rr$, $p\in(0,\infty)$, $q\in(0,\infty]$, and $W$ be an
$A_p(\rn,\mathbb{C}^m)$-matrix weight. In this article, we
first consider
other real-variable
characterizations of $\dot{F}^{\alpha,q}_{p}(W)$,
including its characterizations via the Peetre maximal
function, the Lusin
area function, and the Littlewood--Paley
$g_\lambda^{*}-$function, respectively, in Theorems \ref{pmf}, \ref{lusin},
and \ref{glambda} below.
We should point out that the main strategy used in \cite{U,P}
to establish these real-variable characterizations
of classical Triebel--Lizorkin spaces is based on a technique of the
application of
the Fefferman--Stein vector-valued maximal inequality.
However, since
the matrix-weighted Fefferman--Stein vector-valued maximal inequality is still
unknown so far,
it follows that the approach used in \cite{U,P} is no longer
feasible for
matrix-weighted Triebel--Lizorkin spaces. To overcome
these obstacles,
we borrow some ideas from \cite{FR} and
introduce both the Peetre maximal function and the Littlewood--Paley function
quasi-norms in terms of  reducing operators associated to
$W$ [see
\eqref{pq} and \eqref{gq} below]. Then the problem can
be reduced to study the equivalence between the quasi-norms
of Triebel--Lizorkin spaces in terms of reducing operators of
$W$ in Definition \ref{AQ} below and the corresponding
Peetre maximal function or the corresponding Littlewood--Paley $g_{\lambda}^{*}$-function
quasi-norm, respectively, in \eqref{pq} and \eqref{gq} below,
which allows us to use the Fefferman--Stein vector-valued
maximal inequality
in $L^p(\rn)$ to solve the problem.
As an application of the
Littlewood--Paley characterization of $\dot{F}^{\alpha,q}_{p}(W)$,
we obtain, in Theorem \ref{fourier multiplier} below, the boundedness
of  Fourier multipliers on $\dot{F}^{\alpha,q}_{p}(W)$ under the
assumption of the H\"ormander condition [see \eqref{Hormander} below].

To be precise, the remainder of this article is
organized as follows.

In Section \ref{s2}, we first recall some concepts
concerning the matrix
weight $W$, the $A_p(\rn,\mathbb{C}^m)$-matrix weight
condition, and the reducing operator of $W$. Then we recall
some known properties and also give some new properties of both
$A_p(\rn,\mathbb{C}^m)$-matrix weights and reducing operators of
matrix weights, which play a key role in the
proof of the whole article.

In Section \ref{s4}, we establish some real-variable
characterizations of $\dot{F}^{\alpha,q}_{p}(W)$. We
first characterize
the matrix-weighted
 Triebel--Lizorkin
space $\dot{F}^{\alpha,q}_{p}(W)$ for any $\alpha\in\rr$, $p\in(0,\infty)$,
$q\in(0,\infty]$, and $W\in A_p(\rn,\mathbb{C}^m)$ in terms of the  Peetre maximal function
(see Theorem \ref{pmf} below).
By introducing the Peetre maximal function with the
reducing operator relating
to the matrix weight, we obtain both the Lusin area function and
the Littlewood--Paley $g_\lambda^*$-function
characterizations of
matrix-weighted Triebel--Lizorkin spaces (see Theorems
\ref{lusin} and
\ref{glambda} below).

In Section \ref{s6}, we prove the boundedness of Fourier multipliers on
$\dot{F}^{\alpha,q}_{p}(W)$ (see Theorem \ref{fourier multiplier} below) under
the assumption of the H\"ormander condition
[see \eqref{Hormander} below for
its definition], which is an application of
the Littlewood--Paley characterization
of $\dot{F}^{\alpha,q}_{p}(W)$ with $\alpha\in\rr$,
$p\in(0,\infty)$,
$q\in(0,\infty]$, and $W$ being an $A_p(\mathbb{R}^n,
\mathbb{C}^m)$-matrix weight.

Finally, we make some conventions on notation.
We use the symbol $f\lesssim g$ to denote that there exists a positive
constant $C$ such that $f\leq C g$. The symbol $f\sim
g$ is used as an
abbreviation of $f\lesssim g\lesssim f$. If $f\leq
Cg$ and $g=h$ or $g\leq h$,
we then write $f\lesssim g\sim h$ or $f\lesssim
g\lesssim h$, rather than
$f\lesssim g=h$ or $f\lesssim g\leq h$. Let
$\mathbb{N}:=\{1,\,2,\,\dots\}$,
$\zz_+:=\mathbb{N}\cup\{0\}$, and $\zz_+^n:=(\zz_+)^n$.
For any multi-index
$\gamma:=(\gamma_1,\,\dots,\,\gamma_n)\in\zz_+^n$ and any
$x:=(x_1,\ldots,x_n)\in\rn$,
let $|\gamma|:=\gamma_1+\cdots+\gamma_n$, $x^{\gamma}
:=x_1^{\gamma_1}\cdots x_n^{\gamma_n}$,
and $\partial^\gamma:=(\frac{\partial}{\partial x_1})^{\gamma_1}
\cdots(\frac{\partial}{\partial x_n})^{\gamma_n}$. For
any index $p\in[1,\infty]$,
we use $p'$ to denote its \emph{conjugate index}, namely,
$\frac{1}{p}+\frac{1}{p'}=1$.
In addition, for any measurable set $F\subset\rn$,
we denote by $\mathbf{1}_F$ its
\emph{characteristic function}. We use the notation
$\langle f, g\rangle$
to denote a
pairing which is linear in $f$ and $g$; when this pairing is between a
distribution $f$
and a test function $g$, then $\langle f, g\rangle=f(g)$. We also use
the notation
$(\vec{x},\vec{y})$ to denote the inner product of
$\vec{x},\vec{y}\in\mathbb{C}^m$.
For any $f\in L^1_{\loc}(\rn)$ and any measurable set $E\subset\rn$, let
$$\fint_{E} f(x) \,dx:=\frac{1}{|E|}\int_E f(x)\,dx.$$
For any $s\in\rr$,
we use the
symbol $\lfloor s\rfloor$ to denote the \emph{largest integer
not greater than $s$}.
For any measurable function $g$ and any $x\in\rn$, let
$\widetilde{g}(x):=g(-x)$.
For any $x\in\rn$ and $r\in(0,\infty)$, let
$B(x,r):=\{y\in\rn:\ \ |x-y|<r\}$ be the
ball with center $x$ and radius $r$. Furthermore,
for any $a\in(0,\infty)$ and
any ball $B:=B(x_B,r_B)$ in $\rn$ with $x_B\in\rn$
and $r_B\in(0,\infty)$,
let $a B:=B(x_B,a r_B)$. We also use $\mathbf{0}$ to denote
the origin of $\rn$.

Let $\cs(\rn)$ be the \emph{space
of all Schwartz functions} on $\rn$, equipped with the classical
topology determined by a well-known countable family of norms, and $\cs'(\rn)$ its
\emph{topological dual space} [namely, the set of
all continuous linear
functionals on $\cs(\rn)$], equipped with the weak $*$-topology.
Following Triebel,
we let
$$\cs_{\infty}(\rn):=\left\{\varphi\in\cs(\rn): \int_\rn
\varphi(x)x^\gamma dx=0\
\text{for all multi-indices}\ \gamma\in\zz_+^n \right\}$$
and consider $\cs_{\infty}(\rn)$ as a subspace of $\cs(\rn)$,
including its topology. Use $\cs'_{\infty}(\rn)$ to denote the
\emph{topological dual space} of $\cs_{\infty}(\rn)$, namely,
the set of all continuous linear functionals on $\cs_{\infty}(\rn)$.
We also equip
$\cs'_{\infty}(\rn)$ with the weak $*$-topology.
Let $\cp(\rn)$ be the
\emph{set of all polynomials} on $\rn$. It is well known that
$\cs'_{\infty}(\rn)=\cs'(\rn)/\cp(\rn)$ as topological spaces;
see, for instance,
\cite{G2,S}.

\section{Matrix-Weighted Triebel--Lizorkin Spaces
$\dot{F}^{\alpha,q}_{p}(W)$\label{s2}}

In this section, we present some basic definitions and results of
matrix-weighted Triebel--Lizorkin spaces via two subsections.
In Subsection \ref{s21},
we recall the concepts of the matrix weight,  the $A_p(\rn,
\mathbb{C}^m)$-matrix weight,
and their properties. In Subsection \ref{s22}, we
present both the definition  and also
some basic properties
of  matrix-weighted Triebel--Lizorkin spaces.

\subsection{$A_p(\rn,\mathbb{C}^m)$-Matrix Weights\label{s21}}
In this section,
we  recall the concepts of the matrix weight, the
matrix $A_p(\rn,\mathbb{C}^m)$ condition,
and the reducing operator of $W$.
Furthermore, we present their basic properties.
We begin with recalling the concept of the matrix weight
(see, for instance, \cite{r03,t}).
In what follows, for any $\vec{z}:=(z_1,\dots,z_m)^T\in\mathbb{C}^m$,
let $|\vec{z}|:=(\sum_{j=1}^{m}|z_j|^2)^{1/2}$, where  $T$
denotes the \emph{transpose} of
the row vector.
\begin{definition}\label{lian}
Let $m\in\nn$.
An $m\times m$ complex-valued matrix $A$
is said to be \emph{nonnegative definite} if,
for any
$\vec{z}\in\mathbb{C}^m\setminus\{\mathbf{0}\}$,
$(A\vec{z},\vec{z})\geq0$.
An $m\times m$ complex-valued matrix $A$
is said to be \emph{positive definite} if,
for any $\vec{z}\in\mathbb{C}^m\setminus\{\mathbf{0}\}$,
$(A\vec{z},\vec{z})>0$.
The set of
all  nonnegative definite $m\times m$
complex-valued matrices is denoted by $M_m(\mathbb{C})$.
Furthermore, the \emph{operator
norm of a matrix} $A$ is defined by setting
$$\lf\|A\r\|:=\sup_{\vec{z}\in\mathbb{C}^m\setminus
\{\mathbf{0}\}}\frac{|A\vec{z|}}{|\vec{z}|}.$$
\end{definition}

\begin{definition}\label{bn}
Let $m\in\nn$ and
$W:\ \rn\to M_m(\mathbb{C})$
satisfy that every entry of $W$ is
a measurable function.
The map $W$ is called a
\emph{matrix weight} from $\rn$ to $M_m(\mathbb{C})$
if $W(x)$ is invertible for almost every $x\in\rn$.
\end{definition}

The following definition is a part of \cite[Definition 1.2]{h}.

\begin{definition}\label{lian2}
Let $m\in\nn$ and $A$ be
a positive definite $m\times m$
complex-valued matrix satisfying that
there exists an invertible
$m\times m$ complex-valued matrix $P$
and a diagonal matrix $\,\mathrm{diag}\,(\lambda_1,\ldots,\lambda_m)$,
with $\{\lambda_1,\ldots,\lambda_m\}\subset\mathbb{R}_+$,
such that
$A=P\,\mathrm{diag}\,(\lambda_1,\ldots,\lambda_m)P^{-1}$.
Then, for any $\alpha\in\rr$, let
$$
A^\alpha:=P\,\mathrm{diag}\,(\lambda_1^\alpha,\ldots,
\lambda_m^\alpha)P^{-1}.
$$
\end{definition}

\begin{remark}\label{ilian}
Let $A$ be a positive definite $m\times m$
complex-valued matrix in Definition \ref{lian2}.
By \cite[Theorem 4.1.5]{hj},
we find that the decomposition of $A$ in Definition \ref{lian2}
exists.
Furthermore, from \cite[Problem 1.1]{h}
(see also \cite[p.\,407]{hj}),
we deduce that, for any $\alpha\in(0,\infty)$,
$A^\alpha$ in Definition \ref{lian2} is
independent of the choice of the
invertible $m\times m$ complex-valued matrix $P$
and the diagonal matrix $\mathrm{diag}(\lambda_1,\ldots,\lambda_m)$
with $\{\lambda_1,\ldots,\lambda_m\}\subset\mathbb{C}$.
\end{remark}

In what follows, let $\mathcal{Q}:=\{\text{all cubes } Q\subset\rn\}$,
here and thereafter, a \emph{cube} means its edges parallel to the coordinate axis with a
finite and positive edge length which is not necessary to be open.
The following definition
is just
\cite[Definition 3.2]{R}
and \cite[p. 1226, (1.1)]{FR0}.

\begin{definition}\label{Ap}
Let $m\in\nn$ and $p\in(0,\infty)$.
An \emph{$A_p(\rn,\mathbb{C}^m)$-matrix weight} $W$, denoted by $W\in A_p(\rn,\mathbb{C}^m)$, is a
matrix weight from $\rn$ to $M_m(\mathbb{C})$
satisfying that,
when $p\in(1,\infty)$,
\begin{equation*}
\sup_{Q\in\mathcal{Q}}
\frac{1}{|Q|}\int_{Q}\left[\frac{1}{|Q|}\int_{Q}
\lf\|W^{1/p}(x)W^{-1/p}(y)\r\|^{p'}\,dy\right]^{p/p'}
\,dx<\infty,
\end{equation*}
where $\|\cdot\|$ denotes the operator norm of a matrix
and, when $p\in(0,1]$,
\begin{equation*}
\sup_{Q\in\mathcal{Q}}\mathop{\esssup}\limits_{y\in {Q}}
\frac{1}{|Q|}\int_{Q}\lf\|W^{1/p}(x)W^{-1/p}(y)\r\|^{p}
\,dx<\infty.
\end{equation*}
\end{definition}

\begin{remark}\label{ww}
When $p\in[1,\infty)$ and $m=1$,
the $A_p(\rn,\mathbb{C}^m)$-matrix weight in Definition \ref{Ap}
coincides with the classical $A_p(\rn)$-weight
(see Definition \ref{weight} below for its definition).
\end{remark}

The following result about the matrix weight is just
\cite[Corollary 3.3]{r03}.
\begin{lemma}\label{matrix weight}
Let $W$ be a matrix weight
from $\rn$ to $M_m(\mathbb{C})$,
$p\in(1,\infty)$, and $p':=p/(p-1)$.
Then the following statements are equivalent:
\begin{enumerate}
\item [\rm{(i)}] $W\in A_p(\rn,\mathbb{C}^m)$;
\item [\rm{(ii)}] $W^{-p'/p}\in A_{p'}(\rn,\mathbb{C}^m)$.
\end{enumerate}
\end{lemma}

Now, we recall the concept of
the classical $A_p(\rn)$-weight
(see, for instance, \cite{G1}).

\begin{definition}\label{weight}
An \emph{$A_p(\rn)$-weight} $\omega$,
with $p\in[1,\infty)$, is a
locally integrable
and nonnegative function on $\rn$
satisfying that,
when $p\in(1,\infty)$,
\begin{equation*}
\sup_{Q\in\mathcal{Q}}\lf[\frac1{|Q|}\int_Q\omega(x)\,dx\r]
\lf[\frac1{|Q|}
\int_Q\lf\{\omega(x)\r\}
^{\frac1{1-p}}\,dx\r]^{p-1}<\infty
\end{equation*}
and, when $p=1$,
$$
\sup_{Q\in\mathcal{Q}}\frac{1}{|Q|}
\int_Q\omega(x)\,dx\lf[\lf\|\omega^{-1}\r\|_
{L^\infty(Q)}\r]<\infty.
$$
Define $A_\infty(\rn):=\bigcup_{p\in[1,\infty)}A_p(\rn)$.
\end{definition}
From both \cite[Corollary 2.2]{G} and
\cite[Lemma 2.1]{FR0},
we deduce the following lemma;  we omit the details here.
\begin{lemma}\label{scalar weight}
Let $p\in(0,\infty)$,
$W\in A_p(\rn,\mathbb{C}^m)$,
and $w_{\vec{y}}(x):=|W^{1/p}(x)\vec{y}|^p$
for any $x\in\rn$ and
any given $\vec{y}\in \mathbb{C}^m$.
Then, for any given
$\vec{y}\in\mathbb{C}^m\setminus\{\mathbf{0}\}$,
$w_{\vec{y}}\in A_1(\rn)$
if $p\in(0,1]$, and
$w_{\vec{y}}\in A_p(\rn)$
if $p\in(1,\infty)$.
\end{lemma}
If $p\in(1,\infty)$, the following corollary is just
\cite[Corollary 2.3]{G}.
For the convenience of
the reader, we present some details of its proof.
\begin{corollary}\label{norm weight}
Let $p\in(0,\infty)$
and $W\in A_p(\rn,\mathbb{C}^m)$.
Then
$\|W^{1/p}\|^p\in A_1(\rn)$
if $p\in(0,1]$, and
$\|W^{1/p}\|^p\in A_p(\rn)$
if $p\in(1,\infty)$.
\end{corollary}

\begin{proof}
By \cite[Lemma 3.2]{r03},
we conclude that, for any given $p\in(0,\infty)$ and
for any $x\in\rn$,
$$
\lf\|W^{1/p}(x)\r\|^p\sim\sum_{i=1}^{m}
\lf|W^{1/p}(x)\vec{e_i}\r|^p,
$$
where
$\{\vec{e_1},\,\dots,\,\vec{e_m}\}$
is the standard unit basis of $\mathbb{C}^m$.
Then, by Lemma \ref{scalar weight}, we conclude that,
for any $i\in\{1,\,\dots,\,m\}$,
$|W^{1/p}\vec{e_i}|^p$ is an $A_1(\rn)$-weight
if $p\in(0,1]$, and
$|W^{1/p}\vec{e_i}|^p$ is an $A_p(\rn)$-weight
if $p\in(1,\infty)$,
therefore, their finite sum is as well. This
finishes the proof of Corollary \ref{norm weight}.
\end{proof}

The following definition
comes from \cite[p. 1230]{FR0}.

\begin{definition}\label{doubling weight}
Let $p\in(0,\infty)$. A non-zero matrix weight $W$
is called
\emph{a doubling matrix weight of order $p$}
if there exists a positive
constant $C$ such that,
for any cube $Q\subset\rn$
and any $\vec{z}\in\mathbb{C}^m$,
\begin{equation}\label{wsd}
\int_{2Q}\lf|W^{1/p}(x)\vec{z}\r|^p\,dx\leq C
\int_{Q}\lf|W^{1/p}(x)\vec{z}\r|^p\,dx,
\end{equation}
where $2Q$ denotes the cube concentric with $Q$ and
having twice the edge length of $Q$.
Let $$\beta:=
\min\left\{\beta\in(0,\infty):\ \
\text{\eqref{wsd} holds with $C=2^\beta$}\right\}.$$
Then $\beta$ is called the
\emph{doubling exponent of the doubling matrix
weight $W$ of order $p$}.
For simplicity,  such a $\beta$ is also called the
\emph{doubling exponent of $W$}.
\end{definition}

The following lemma can be found in \cite[Lemma 2.2]{FR}.
\begin{lemma}\label{apd}
Let $p\in(0,\infty)$
and $W\in A_p(\rn,\mathbb{C}^m)$.
Then $W$ is a doubling matrix weight of order $p$.
\end{lemma}

\begin{remark}\label{ball}
	It is easy to see that,
	if we replace any cube $Q$ with
	any ball $B\subset\rn$ in Definitions \ref{Ap}
	and \ref{doubling weight}, then
	Lemmas \ref{matrix weight}
	and \ref{apd} still hold true.
\end{remark}

In what follows,
for any $j\in\mathbb{Z}$
and $k:=(k_1,\,\dots,k_n)\in\mathbb{Z}^{n}$,
let $Q_{jk}:=\prod_{i=1}^n 2^{-j}[k_i,k_i+1)$,
$\mathcal{D}:=\{Q_{jk}:j\in\mathbb{Z},k\in\mathbb{Z}^{n}\}$,
and
\begin{equation}\label{dj}
\mathcal{D}_{j}:=\{Q_{jk}:k\in{\mathbb{Z}}^{n}\}.
\end{equation}

\begin{definition}\label{yue}
Let $m\in\nn$,
$p\in(0,\infty)$,
and $W$ be a matrix weight
from $\rn$ to $M_m(\mathbb{C})$.
A sequence
$\{A_Q^{(W)}\}_{Q\in\mathcal{D}}$
of positive definite
$m\times m$ matrices
is called
\emph{a sequence of
reducing operators of order $p$ for $W$}
if there exist positive constants
$C_1$ and $C_2$
such that,
for any $\vec{z}\in\mathbb{C}^m$
and $Q\in\mathcal{D}$,
$$
C_1\lf|A_Q^{(W)}\vec{z}\r|
\leq\lf[\frac{1}{|Q|}\int_Q
\lf|W^{1/p}(x)\vec{z}\r|^pdx\r]^{1/p}
\leq C_2\lf|A_Q^{(W)}\vec{z}\r|.
$$
\end{definition}
	
For simplicity,
the sequence $\{A_Q^{(W)}\}_{Q\in\mathcal{D}}$
of reducing operators of order $p$ for $W$
is denoted by $\{A_Q\}_{Q\in\mathcal{D}}$.

\begin{remark}\label{aqe}
Let $m\in\nn$.
From \cite[Proposition 1.2]{G}
and \cite[p. 1237]{FR0},
we deduce that,
for any $p\in(0,\infty)$ and any
matrix weight $W$
from $\rn$ to $M_m(\mathbb{C})$,
a sequence of
reducing operators of
order $p$ for $W$
in Definition \ref{yue} exists.
\end{remark}

The following lemmas are
respectively a part of \cite[Lemmas 3.2 and 3.3]{FR}.

\begin{lemma}\label{a123}
Let $p\in(1,\infty)$,
$p':=p/(p-1)$,
$W\in A_p(\rn,\mathbb{C}^m)$,
and $\{A_Q\}_{Q\in\mathcal{D}}$
be a sequence of reducing operators
of order $p$ for $W$.
Then there exists a  $\delta(W)\in(0,\infty)$
such that, for any $\eta\in(0, p'+\delta(W))$,
\begin{equation*}
\sup_{Q\in\mathcal{D}}
\frac{1}{|Q|}\int_{Q}\lf\|A_QW^{-1/p}(x)\r\|^{\eta}
\,dx<\infty .
\end{equation*}
\end{lemma}

\begin{lemma}\label{ainf}
Let $p\in(0,1]$,
$W\in A_p(\rn,\mathbb{C}^m)$,
and $\{A_Q\}_{Q\in\mathcal{D}}$
be a sequence of
reducing operators
of order $p$ for $W$.
Then
\begin{equation*}
\sup_{Q\in\mathcal{D}}
\mathop{\esssup}_{x\in Q}
\lf\|A_QW^{-1/p}(x)\r\|<\infty.
\end{equation*}
\end{lemma}

\subsection{Matrix-Weighted Triebel--Lizorkin Spaces\label{s22}}

In this section, we begin with recalling the concepts
of both matrix-weighted
Triebel--Lizorkin spaces and sequence matrix-weighted
Triebel--Lizorkin spaces.
Then we prove the rationality of Definitions
\ref{FBW} and \ref{AQ}.

In what follows,
for any $m\in\nn$,
let
\begin{align*}
[\cs'_{\infty}(\rn)]^m:=
\lf\{\vec{f}:=(f_1,\,\ldots,\,f_m)^T:\ \
\text{for any}\ i\in\{1,\,\ldots,\,m\},
\ f_i\in \cs'_{\infty}(\rn)\r\}.
\end{align*}
For any
$\vec{f}:=(f_1,\,\ldots,\,f_m)^T
\in[\cs'_{\infty}(\rn)]^m$
and $\varphi\in\cs_{\infty}(\rn)$,
let
\begin{align*}
\varphi\ast\vec{f}:=
\lf(\varphi\ast f_1,\,\ldots,\,\varphi\ast f_m\r)^T
\end{align*}
and
$\varphi_j(\cdot):=2^{jn}\varphi(2^j \cdot)$
with $j\in\zz$.
For any $\phi\in\cs(\rn)$, $\widehat\phi$
denotes its \emph{Fourier transform}
which is defined by setting,
for any $\xi\in\rn$,
$$
\widehat\phi(\xi):=(2\pi)^{-n/2}\int_\rn \phi(x)e^{-ix\xi}\,dx.
$$
For any $f\in\cs'(\rn)$,
$\widehat f$ is defined by setting,
for any $\varphi\in\mathcal{S}(\rn)$,
$\langle\widehat f,\varphi\rangle:=\langle f,\widehat{\varphi}\rangle$;
also, for any $f\in\mathcal{S}(\rn)$
[resp., $\mathcal{S}'(\rn)$],
$f^{\vee}$ denotes its \emph{inverse Fourier transform},
$$f^{\vee}(\cdot):=(2\pi)^{n/2}\int_{\rn}\widehat{f}(x)e^{ix\cdot}
\,dx$$
[resp., $\langle f^{\vee},\varphi\rangle:=\langle f,
\varphi^{\vee}\rangle$
for any $\varphi\in\cs(\rn)$].
For any $\varphi\in\cs(\rn)$, $\supp\widehat{\varphi}:=
\{x\in\rn:\widehat{\varphi}(x)\neq 0\}$
and, for any $f\in\cs'(\rn)$,
$$
\supp f:=\cap\lf\{\text{closed set }
K\subset \rn:\langle f,\,\varphi\rangle=0
\text{ if }\varphi\in\cs(\rn)\,\text{and}\,\supp\varphi
\subset \rn\setminus K \r\},
$$
which can be found in \cite[Definition 2.3.16]{G1}.

\begin{definition}\label{FBW}
Let $m\in\mathbb N$,
$\alpha\in\mathbb{R}$,
$p\in(0,\infty)$,
$q\in(0,\infty]$,
$W$ be a matrix weight
from $\rn$ to $M_m(\mathbb{C})$,
and $\varphi\in\cs(\rn)$.
Furthermore, assume that
\begin{itemize}
\item [(T1)] for any $x\in\rn\setminus\{\mathbf{0}\}$,
there exists an $l\in\zz$
such that $\widehat{\varphi}(2^l x)\neq 0$,
\item [(T2)] $\overline{\supp\widehat{\varphi}}
    \subset\{x\in\rn:|x|<\pi\}$
is  bounded away from the origin.
\end{itemize}
Then
the \emph{matrix-weighted Triebel--Lizorkin space
$\dot{F}_{p,\varphi}^{\alpha,q}(W)$}
is defined by setting
\begin{equation*}
\dot{F}_{p,\varphi}^{\alpha,q}(W):=
\left\{\vec{f}\in\lf[\cs'_{\infty}(\rn)\r]^m: \ \
\lf\|\vec f\r\|_
{\dot{F}_{p,\varphi}^{\alpha,q}(W)}
<\infty\right\},
\end{equation*}
where
$$
\lf\|\vec{f}\r\|_
{\dot{F}_{p,\varphi}^{\alpha,q}(W)}:=
\lf\|\left[\sum_{j\in\mathbb{Z}}
\lf|2^{j\alpha}W^{1/p}\lf(\varphi_{j}\ast\vec{f}\r)\r|^{q}
\right]^{1/q}\r\|_{L^{p}(\rn)}$$
with suitable modification made when $q=\infty$.
\end{definition}

\begin{remark}\label{r26}
\begin{itemize}
\item [(i)]	Observe that,
if $\widehat{\varphi}(x)>0$
for any $\{x\in\rn:\ \epsilon\leq|x|\leq\pi-b\}$,
where $b\in(0,\pi-1]$ and $\epsilon\in(0,(\pi-b)/2)$,
then $\varphi$ automatically satisfies that,
for any $x\in\rn\setminus\{\mathbf{0}\}$,
there exists an $\ell\in\zz$
such that $\widehat{\varphi}(2^{\ell}x)\neq 0$
and hence, in this case,
the assumption (T1) in Definition \ref{FBW} is superfluous;
see \cite[Lemma 3.18 and Remark 3.19]{WYYZ} for the details.		
\item[(ii)] Let
$\alpha\in\mathbb{R}$,
$p\in(0,\infty)$, and
$q\in(0,\infty]$ be the same as in  Definition \ref{FBW}.
If $m=1$ and $\varphi\in\cs(\rn)$
satisfies both
\begin{equation}\label{2.1mm}
\supp\widehat{\varphi}\subset\{\xi\in\rn:1/2\leq|\xi|\leq 2\}
\end{equation}
and
\begin{equation}\label{2.2mm}
\left|\widehat{\varphi}(\xi)\right|\geq c>0
\end{equation}
when $3/5\leq|\xi|\leq 5/3$ with $c$ being a positive constant independent of
$\xi$,
then $\dot{F}_{p,\varphi}^{\alpha,q}(W)$ in Definition \ref{FBW}
is independent of the choice of $\varphi$ and it
coincides with the weighted Triebel--Lizorkin space
$\dot{F}_{p}^{\alpha,q}(\omega)$.
For more details on weighted Triebel--Lizorkin spaces,
we refer the reader to
\cite{b,bpt,Y2014}.
\item [(iii)]  Let
$\alpha\in\mathbb{R}$,
$p\in(0,\infty)$, and
$q\in(0,\infty]$ be the same as in  Definition \ref{FBW}.
If $m=1$, $W=1$,
and $\varphi\in\cs(\rn)$ satisfies both \eqref{2.1mm} and \eqref{2.2mm},
then $\dot{F}_{p,\varphi}^{\alpha,q}(W)$ in Definition
\ref{FBW} is
independent of the choice of  $\varphi$
(see, for instance, \cite[Remark 2.6]{fj2}) and it
coincides with the Triebel--Lizorkin space
$\dot{F}_{p}^{\alpha,q}$ in
\cite[p. 46]{fj2}.
\end{itemize}
\end{remark}

\begin{definition}\label{AQ}
Let $m\in\mathbb N$,
$\alpha\in\mathbb{R}$,
$p\in(0,\infty)$, $q\in(0,\infty]$,
$\{A_Q\}_{Q\in\mathcal{D}}$
be a sequence of $m\times m$
nonnegative definite matrices,
and $\varphi\in\cs(\rn)$ satisfy both (T1) and (T2)
of Definition \ref{FBW}.
The \emph{$\{A_Q\}$-Triebel--Lizorkin space
$\dot{F}_{p,\varphi}^{\alpha,q}(\{A_Q\})$}
is defined by setting
\begin{equation*}
\dot{F}_{p,\varphi}^{\alpha,q}(\{A_Q\}):=
\left\{\vec{f}\in\lf[\cs'_{\infty}(\rn)\r]^m: \ \
\lf\|\vec f\r\|_
{\dot{F}_{p,\varphi}^{\alpha,q}(\{A_Q\})}
<\infty\right\},
\end{equation*}
where
$$
\lf\|\vec{f}\r\|_
{\dot{F}_{p,\varphi}^{\alpha,q}(\{A_Q\})}:=
\lf\|\left[\sum_{j\in\mathbb{Z}}\sum_{Q\in\mathcal{D}_j}
\left(2^{j\alpha}\left|A_Q\varphi_{j}*\vec{f}\,
\right|\mathbf{1}_Q\right)^{q}
\right]^{1/q}\r\|_{L^{p}(\rn)}
$$
with suitable modification made when $q=\infty$.
\end{definition}

\begin{remark}\label{r28}
Let $m\in\mathbb N$,
$\alpha\in\mathbb{R}$,
$p\in(0,\infty)$, $q\in(0,\infty]$,
and $\{A_Q\}_{Q\in\mathcal{D}}$ be a
sequence of nonnegative definite matrices in Definition
\ref{AQ}.
If $m=1$, $A_Q=1$
for any $Q\in\mathcal{D}$,
and $\varphi\in\cs(\rn)$ satisfies both \eqref{2.1mm} and \eqref{2.2mm},
then $\dot{F}_{p,\varphi}^{\alpha,q}(\{A_Q\})$
in Definition \ref{AQ}
is independent of the choice of $\varphi$
(see, for instance, \cite[Remark 2.6]{fj2}) and it
coincides with the Triebel--Lizorkin space
$\dot{F}_{p}^{\alpha,q}$ in
\cite[p. 46]{fj2}.
\end{remark}

To prove the rationality of Definitions \ref{FBW}
 and \ref{AQ}, we first
recall some classical results (Lemmas \ref{CE}, \ref{bh}, and \ref{fz}), which are
 just \cite[Lemma 3.18]{WYYZ},
\cite[Lemma 2.1]{YY} (see also \cite[Lemma 2.1]{FJ}),
and \cite[Lemma 2.1]{YY1}, respectively.

\begin{lemma}\label{CE}
Let $\varphi$ be a Schwartz function
satisfying that,
for any $x\in\rn \setminus\{\mathbf{0} \}$,
there exists an $l\in\mathbb Z$
such that $\widehat\varphi(2^l x)\neq 0$.
Then there exists a
$\psi\in\cs(\rn)$
such that
$\widehat{\psi}\in\mathcal{C}_{\c}^{\infty}(\rn)$
with its support away from origin,
$\widehat{\varphi}\widehat{\psi}\geq 0$,
and

\begin{equation}\label{cf}
\sum_{j\in\mathbb Z}
\widehat{\varphi}(2^{-j}x)
\widehat{\psi}(2^{-j}x)=1
\end{equation}
for any $x\in\rn \setminus\{\mathbf{0} \}$.
\end{lemma}
In what follows,
for any $Q\in\mathcal{D}$,
let  $\ell(Q)$ denote its
\emph{edge length}
and $x_Q$ its
\emph{lower left corner}.

\begin{lemma}\label{bh}
Let $\varphi$ and $\psi$
be the Schwartz functions such that
$\overline{\supp\widehat{\varphi}},
\overline{\supp\widehat{\psi}}\subset\{x\in\rn:|x|<\pi\}$
are  bounded away from the origin and, furthermore,
for any $x\in\rn\setminus\{\mathbf{0}\}$,
\begin{equation*}
\sum_{j\in\zz}\widehat{\varphi}(2^{-j}x)
\widehat{\psi}(2^{-j}x)=1.
\end{equation*}
Then, for any $f\in\cs'_\infty(\rn)$,
\begin{equation*}
f=\sum_{j\in\zz}2^{-jn}\sum_{k\in\zz^n}
\varphi_j\ast f(2^{-j}k)\psi_j(\cdot-2^{-j}k)
\end{equation*}
converges in $\cs'_\infty(\rn)$.
\end{lemma}

\begin{lemma}\label{fz}
Let $\varphi$ and $\psi$ be the Schwartz functions
satisfying \eqref{cf} and that
both $\supp\widehat{\varphi},\supp\widehat{\psi}$
are compact and bounded away from the origin.
Then, for any $f\in\cs_\infty(\rn)$,
\begin{equation}\label{cin}
f=\sum_{j\in\zz}\varphi_j\ast\psi_j\ast f
\end{equation}
holds true in $\cs_\infty(\rn)$.
Moreover,
for any $f\in\cs'_\infty(\rn)$,
\eqref{cin} also holds true in $\cs_\infty'(\rn).$
\end{lemma}

\begin{lemma}\label{independent}
Let $\alpha\in\mathbb{R}$,
$p\in(0,\infty)$, $q\in(0,\infty]$,
$W\in A_p(\rn,\mathbb{C}^m)$,
$\{A_Q^{(W)}\}_{Q\in\mathcal{D}}$
be a sequence of reducing operators
of order $p$ for $W$,
and $\varphi,\psi\in\cs(\rn)$
satisfy both (T1) and (T2) of Definition \ref{FBW}.
Then,
for any $\vec{f}\in[\cs'_\infty(\rn)]^m$,
\begin{equation*}
\lf\|\vec{f}\r\|_
{\dot{F}_{p,\varphi}^{\alpha,q}(W)}
\sim\lf\|\vec{f}\r\|_{\dot{F}_{p,\psi}^{\alpha,q}(W)}
\end{equation*}
and
\begin{equation*}
\lf\|\vec{f}\r\|_
{\dot{F}_{p,\varphi}^{\alpha,q}(\{A_Q\})}
\sim\lf\|\vec{f}\r\|_
{\dot{F}_{p,\psi}^{\alpha,q}(\{A_Q\})},
\end{equation*}
where the positive equivalence
constants are independent of $\vec{f}$.
\end{lemma}

\begin{proof}
Let  $\varphi$ and $\psi$ be the same as
in Definitions \ref{FBW}
and \ref{AQ}.
By Lemma \ref{CE}, we find that
there exist $\gamma,\eta\in\cs(\rn)$
satisfying
$\widehat{\gamma},\widehat{\eta}\in C_{\c}^\infty(\rn)$
with their supports away from the
origin and \eqref{cf} via $\varphi$ and $\psi$ replaced, respectively, by
$\varphi$ and $\gamma$
or by $\psi$ and $\eta$.
Using this and
Lemma \ref{bh},
and repeating the proof of
\cite[Theorems 1.1 and 2.3]{FR},
we then finish the proof
of Lemma \ref{independent}.
\end{proof}

For simplicity, the matrix-weighted Triebel--Lizorkin
space is denoted by
$\dot{F}^{\alpha,q}_{p}(W)$ and the $\{A_Q\}$-Triebel--Lizorkin
space is
denoted by $\dot{F}^{\alpha,q}_{p}(\{A_Q\})$.
By a proof similar to that used in \cite[Theorem 1.1]{FR},
we conclude the following lemma;  we omit the details here.
\begin{lemma}\label{eq}
Let $m\in\nn$, $\alpha\in\rr$,
$p\in(0,\infty)$, $q\in(0,\infty]$,
$W\in A_p(\rn,\mathbb{C}^m)$,
$\{A_Q\}_{Q\in\mathcal{D}}$ be a sequence of
reducing operators
of order $p$ for $W$, and $\varphi\in\cs(\rn)$ satisfy both
(T1) and (T2) of Definition \ref{FBW}.
Then, for any $\vec{f}\in[\cs_\infty'(\rn)]^m$,
$$
\|\vec{f}\|_{\dot{F}^{\alpha,q}_p(W)}\sim
\|\vec{f}\|_{\dot{F}^{\alpha,q}_p(\{A_Q\})},
$$
where the positive equivalence constants
are independent of $\vec{f}$.
\end{lemma}

\section{Real-Variable Characterizations
of $\dot{F}^{\alpha,q}_{p}(W)$}\label{s4}

In this section,
we characterize the spaces
$\dot{F}_{p}^{\alpha,q}(W)$
via the Peetre maximal function, the
Lusin area function, and the
Littlewood--Paley
$g_\lambda^*-$function.

Now, with a modification of
the classical Peetre-type maximal function
in \cite{P2},
we introduce the concept of
the following
\emph{matrix-weighted Peetre-type maximal function}.
Let $p\in(0,\infty)$,
$m\in\nn$,
$W\in A_p(\rn,\mathbb{C}^m)$,
$\varphi\in\cs_{\infty}(\rn)$,
and $\vec{f}\in[\cs_{\infty}'(\rn)]^{m}$.
For any given $j\in\zz$ and
$a\in(0,\infty)$,
and for any $x\in\rn$,
let
\begin{equation}\label{pjd}
\lf(\varphi_j^{*}\vec{f}\r)_{a}^{(W,p)}(x):=
\sup_{y\in\rn}
\frac{|W^{1/p}(x)(\varphi_{j}*\vec f)(y)|}{(1+2^{j}|x-y|)^a}.
\end{equation}

\begin{theorem}\label{pmf}
Let $\alpha\in\mathbb{R}$,
$p\in(0,\infty)$,
$q\in(0,\infty]$,
$W\in A_p(\rn,\mathbb{C}^m)$, and
$$a\in\left(n/\min\{1,p,q\}+\beta/p,\infty\right),$$
where $\beta$ is the doubling exponent of $W$.
Assume that
$\varphi\in\cs(\rn)$
satisfies both (T1) and (T2) of Definition \ref{FBW}.
Then
$\vec{f}\in\dot{F}^{\alpha,q}_{p}(W)$
if and only if
$\vec{f}\in[\cs_\infty'(\rn)]^m$ and
$\|\vec{f}\|_{\dot{F}^{\alpha,q}_{p}(W)}^{\star}<\infty$,
where
\begin{equation}\label{pjida}
\lf\|\vec{f}\r\|^{\star}_
{\dot{F}^{\alpha,q}_{p}(W)}
:=\left\|\left\{\sum_{j\in\mathbb Z}
2^{j\alpha q}\left[\lf(\varphi_{j}^{*}\vec f\r)_{a}^
{(W,p)}\right]^{q}\right\}^{1/q}\right\|
_{L^p(\rn)}
\end{equation}
with usual modification made when $q=\infty$. Moreover,
for any $\vec{f}\in[\cs_{\infty}'(\rn)]^m$,
\begin{equation}\label{pmfd}
\lf\|\vec{f}\r\|_{\dot{F}^{\alpha,q}_{p}(W)}\sim\lf\|\vec{f}
\r\|^{\star}_
{\dot{F}^{\alpha,q}_{p}(W)},
\end{equation}
where the positive equivalence constants are independent of
$\vec{f}$.
\end{theorem}

\begin{remark}\label{pr}
\begin{itemize}
\item [(\rm{i})]
Theorem \ref{pmf} when $m=1$
is a part of \cite[Theorem 3.1]{bpt96}.
\item [(\rm{ii})] Theorem
\ref{pmf} when $m=1$ and $W=1$ is
a part of \cite[Theorem 3.1]{P2}
which is the Peetre maximal function
characterization of
Triebel--Lizorkin spaces.
\end{itemize}
\end{remark}

To show Theorem \ref{pmf},
we first recall the definitions of both
strongly doubling and weakly doubling matrices,
which can be found in \cite[Definition 2.1]{FR}.

\begin{definition}\label{swd}
Let $\{A_Q\}_{Q\in\mathcal{D}}$
be a sequence of positive definite matrices,
$\beta\in(0,\infty)$,
and $p\in(0,\infty)$.
The sequence $\{A_Q\}_{Q\in\mathcal{D}}$
is said to be
\emph{strongly doubling of order $(\beta,p)$}
if there exists a positive constant $C$ such that,
for any $Q,P\in\mathcal{D}$,
\begin{equation}\label{sd}
\lf\|A_QA_P^{-1}\r\|^p
\leq C\max\lf\{\lf[\frac{\ell(P)}{\ell(Q)}\r]^n\, ,\,
\lf[\frac{\ell(Q)}{\ell(P)}\r]^{\beta-n}\r\}
\lf[1+\frac{|x_Q-x_P|}{\max\{\ell(P)\,
, \, \ell(Q)\}}\r]^\beta.
\end{equation}
The sequence $\{A_Q\}_{Q\in\mathcal{D}}$
is said to be
\emph{weakly doubling of order $r\in(0,\infty)$}
if there exists a positive constant $C$ such that,
for any $k,\ell\in\zz^n$ and $j\in\zz$,
\begin{equation}\label{wd}
\lf\|A_{Q_{jk}}A_{Q_{j\ell}}^{-1}\r\|
\leq C\lf(1+|k-\ell|\r)^r,
\end{equation}
where
$Q_{jk}:=\prod_{i=1}^n 2^{-j}[k_i,k_i+1)$
for any $j\in\zz$
and $k:=(k_1,\,\dots,k_n)\in\zz^n$.
\end{definition}

\begin{remark}\label{ws}
In Definition \ref{swd}, a strongly doubling sequence
of order $(\beta,p)$ satisfying \eqref{sd} is also
weakly doubling of order $r:=\beta/p$ satisfying \eqref{wd}
because, when $\ell(P)=\ell(Q)$,
\eqref{wd}  coincides with \eqref{sd}.
\end{remark}

The following lemma explains
the connection between the
doubling weight $W$ and the
doubling sequence
$\{A_Q\}_{Q\in\mathcal{D}}$,
which can be deduced from
\cite[Lemma 2.2]{FR},
Lemma \ref{apd}, and Remark \ref{ws}; we omit
the details.

\begin{lemma}\label{wsz}
Let $p\in(0,\infty)$,
$W\in A_p(\rn,\mathbb{C}^m)$,
and $\{A_Q\}_{Q\in\mathcal{D}}$
be a sequence of reducing operators
of order $p$ for $W$.
Then $\{A_Q\}_{Q\in\mathcal{D}}$
is weakly doubling of order $\frac{\beta}{p}$,
where $\beta$ is the doubling exponent of $W$.
\end{lemma}

The following lemma
is just \cite[(2.8)]{FR}.

\begin{lemma}\label{cy}
Let $\varphi\in\cs(\rn)$ satisfy (T2) of Definition \ref{FBW}.
Suppose that $\{A_Q\}_{Q\in\mathcal{D}}$
is a weakly doubling sequence
of order $r\in(0,\infty)$ of
positive definite matrices.
Then,
for any given $A\in(0,1]$
and $R\in(0,\infty)$,
there exists a positive constant
$C$, depending on both $A$ and $R$,
such that,
for any $j\in\zz$,
$k\in\zz^n$,
and $\vec{f}\in[\cs'_\infty(\rn)]^m$,
\begin{align*}
&\sup_{x\in Q_{jk}}
\lf|A_{Q_{jk}}\lf(\varphi_j\ast\vec{f}\r)(x)\r|^A\\
&\quad\leq C\sum_{\ell\in\zz^n}
\lf(1+|k-\ell|\r)^{-A(R-r)}2^{jn}
\int_{Q_{j\ell}}
\lf|A_{Q_{j\ell}}\varphi_j\ast\vec{f}(s)\r|^A\,ds,
\end{align*}
where
$Q_{jk}:=\prod_{i=1}^n 2^{-j}[k_i,k_i+1)$
for any $j\in\zz$
and $k:=(k_1,\,\dots,k_n)\in\zz^n$.
\end{lemma}

Recall that the
\emph{Hardy--Littlewood maximal operator $\mathcal{M}$}
is defined by setting,
for any locally integrable function
$f$ and any $x\in\rn$,
\begin{equation}\label{maximal function}
\mathcal{M}(f)(x):=\sup_{x\in B}\frac{1}{|B|}\int_B
\lf|f(y)\r|\,dy=\sup_{x\in B} \fint_B \lf|f(y)\r|\,dy,
\end{equation}
where the supremum is
taken over all the balls $B$ of $\rn$ containing $x$.
Denote by the \emph{symbol}
$\mathscr{M}(\rn)$
the set of all the complex-valued
measurable functions on $\rn$.

\begin{lemma}\label{jcf}
Let $\mathcal{M}$ be
the maximal operator
in \eqref{maximal function}
and $\eta>n$.
Then there exists a positive constant
$C$ such that,
for any $j\in\zz$ and $h\in \mathscr{M}(\rn)$,
\begin{equation*}
\sum_{k\in\zz^n}
\sum_{\ell\in\zz^n}
(1+|k-\ell|)^{-\eta}2^{jn}
\int_{Q_{j\ell}}
|h(s)|\,ds\mathbf{1}_{Q_{jk}}
\leq C \mathcal{M}(h).
\end{equation*}
\end{lemma}

\begin{proof}
Observe that, for any given $j\in\zz$
and any $x\in\rn$, it is easy to see that
there exists a unique $k_x\in\zz^n$
such that $x\in Q_{jk_x}$.
Using this, we find that, for any $x\in\rn$,
\begin{align*}
&\sum_{k\in\zz^n}
\sum_{\ell\in\zz^n}
(1+|k-\ell|)^{-\eta}2^{jn}
\int_{Q_{j\ell}}|h(s)|\,ds
\mathbf{1}_{Q_{jk}}(x)\\
&\quad=
\sum_{\ell\in\zz^n}
(1+|k_x-\ell|)^{-\eta}2^{jn}
\int_{Q_{j\ell}}|h(s)|\,ds\\
&\quad=\sum_{\{\ell\in\zz^n:\,|\ell-k_x|\leq 1\}}
(1+|k_x-\ell|)^{-\eta}2^{jn}
\int_{Q_{j\ell}}|h(s)|\,ds
+\sum_{m\in\mathbb{N}}
\sum_{\{\ell\in\zz^n:\,2^{m-1}<|\ell-k_x|\leq 2^m\}}
\cdots\\
&\quad\lesssim\sum_{\{\ell\in\zz^n:\,|\ell-k_x|\leq 1\}}
\int_{Q_{j\ell}}|h(s)|\,ds
+\sum_{m\in\mathbb{N}}
\sum_{\{\ell\in\zz^n:\,2^{m-1}<|\ell-k_x|\leq 2^m\}}
2^{-m\eta}\int_{Q_{j\ell}}|h(s)|\,ds\\
&\quad\sim
\int_{\bigcup_{\{\ell\in\zz^n:\,|\ell-k_x|\leq1\}}
Q_{j\ell}}|h(s)|\,ds+
\sum_{m\in\mathbb{N}}2^{-m\eta}
\int_{\bigcup_{\{\ell\in\zz^n:\,2^{m-1}<|\ell-k_x|
\leq 2^m\}}Q_{j\ell}}
|h(s)|\,ds\\
&\quad\lesssim
\sum_{m\in\zz_+}2^{-m\eta}2^{mn}
\fint_{B_m}|h(s)|\,ds\lesssim\mathcal{M}(h)(x),
\end{align*}
where $B_m$ for any $m\in\zz_+$ is the smallest ball containing both
$x$ and $\bigcup_{\{\ell\in\zz^n:\,|\ell-k_x|\leq 2^m\}} Q_{j\ell}$.
This finishes the proof of Lemma \ref{jcf}.
\end{proof}
Now, we recall the definition of
the space $L^p(\ell^q)$,
which can be found in \cite[p. 14]{T1}.

\begin{definition}\label{lplq}
Let $p\in(0,\infty]$ and $q\in(0,\infty]$.
Then the \emph{space $L^p(\ell^q)$} is defined by setting
$$L^p(\ell^q):=\lf\{\{f_j\}_{j\in\zz}\subset\mathscr{M}(\rn):\
 \ \lf\|\lf\{f_j\r\}_{j\in\zz}\r\|_{L^p(\ell^q)}<\infty \r\},$$
where
\begin{equation*}
\lf\|\lf\{f_j\r\}_{j\in\zz}\r\|_{L^p(\ell^q)}:=
\lf\{\int_\rn\lf[\sum_{j\in\zz}
\lf|f_j(x)\r|^q\r]^{p/q}\,dx\r\}^{1/p}
\end{equation*}
with suitable modifications made when $p=\infty$ or $q=\infty$.
\end{definition}
The following lemma
is just \cite[Corollary 3.8]{FR}.

\begin{lemma}\label{c38}
Let $p\in(0,\infty)$,
$q\in(0,\infty]$,
$W\in A_p(\rn,\mathbb{C}^m)$,
and $\{A_Q\}_{Q\in\mathcal{D}}$
be a sequence of reducing operators of
order $p$ for $W$.
For any $j\in\zz$, $x\in\rn$,
and $f\in L^1_{\loc}(\rn)$,
let
$$
\gamma_j(x):=
\sum_{Q\in\mathcal{D}_j}
\lf\|W^{1/p}(x)A_Q^{-1}\r\|\mathbf{1}_{Q}(x)
$$
and
$$
E_j(f):=
\sum_{Q\in\mathcal{D}_j}
\lf[\fint_Q f(y)\,dy\r]\mathbf{1}_Q.
$$
Then there exists a positive constant $C$
such that,
for any sequence $\{f_j\}_{j\in\zz}$
of measurable functions on $\rn$,
\begin{equation*}
\lf\|\lf\{\gamma_jE_j\lf(f_j\r)\r\}_{j\in\zz}\r\|_{L^p(\ell^q)}
\leq C
\lf\|\lf\{E_j\lf(f_j\r)\r\}_{j\in\zz}\r\|_{L^p(\ell^q)}.
\end{equation*}
\end{lemma}

The following lemma is the
famous Fefferman--Stein vector-valued
maximal inequality,
see \cite[Theorem 1]{FS}.

\begin{lemma}\label{lm1}
Let $p\in(1,\infty)$
and $q\in(1,\infty]$.
Then there exists
a positive constant $C$
such that,
for any sequence
$\{f_j\}_{j\in\zz}\subset\mathscr{M}(\rn)$,
\begin{equation*}
\lf\|\lf\{\sum_{j\in\zz}\lf[\cm \lf(f_j\r)\r]^q\r\}^{1/q}\r\|_{L^p(\rn)}
\leq C
\lf\|\lf[\sum_{j\in\zz}\lf|f_j\r|^q\r]^{1/q}\r\|_{L^p(\rn)},
\end{equation*}
where $\mathcal{M}$ is
the same as in  \eqref{maximal function}.
\end{lemma}

\begin{proof}[Proof of Theorem \ref{pmf}]

Let all the symbols be the same as in  the present theorem.
Then, by the
definition of $(\varphi_{j}^{*}\vec f)_{a}^{(W,p)}$
in \eqref{pjd}, we
find that $$W^{1/p}\lf(\varphi_j\ast\vec{f}\r)\leq
\lf(\varphi_{j}^{*}\vec f\r)_{a}^{(W,p)},$$
which implies that, for any $\vec{f}\in[\cs'_\infty(\rn)]^m$,
$$
\lf\|\vec f\r\|_
{\dot{F}^{\alpha,q}_{p}(W)}
\lesssim
\lf\|\vec f\r\|^{\star}_
{\dot{F}^{\alpha,q}_{p}(W)}.
$$
Thus, to show Theorem \ref{pmf},
it remains to prove that,
for any $\vec{f}\in[\cs'_\infty(\rn)]^m$,
\begin{equation}\label{r45}
\lf\|\vec f\r\|^{\star}
_{\dot{F}^{\alpha,q}_{p}(W)}
\lesssim
\lf\|\vec f\r\|
_{\dot{F}^{\alpha,q}_{p}(W)}.
\end{equation}
Let  $\{A_Q\}_{Q\in\mathcal{D}}$
be a sequence of reducing operators of order $p$ for $W$.
For any $\vec{f}\in[\cs'_\infty(\rn)]^m$,
let
\begin{equation}\label{pq}
\lf\|\vec{f}\r\|
_{\dot{F}_{p}^{\alpha,q}(\{A_Q\})}^{\star}:=
\lf\|\lf[\sum_{j\in\zz}
\sum_{Q\in\mathcal{D}_j}2^{j\alpha q}
\sup_{y\in\rn}
\frac{|A_Q(\varphi_j\ast\vec{f})(y)
|^q}{(1+2^j|\cdot-y|)^{aq}}
\mathbf{1}_Q\r]^{1/q}\r\|_{L^p(\rn)},
\end{equation}
where, for any $j\in\zz$, $\mathcal{D}_j$
is the same as in  \eqref{dj}.
To prove \eqref{r45},
we first show that,
for any $\vec{f}\in[\cs_\infty'(\rn)]^m$,
\begin{equation}\label{21}
\lf\|\vec{f}\r\|
_{\dot{F}_{p}^{\alpha, q}(\{A_Q\})}^{\star}
\lesssim
\lf\|\vec{f}\r\|_{\dot{F}_p^{\alpha, q}(\{A_Q\})}.
\end{equation}
Indeed, by Lemma \ref{wsz}, a geometrical
observation that $1+2^j|x-y|\sim 1+|k-s|$
for any $x\in Q_{jk}$ and $y\in Q_{js}$,
Lemma \ref{cy}, and the fact
that $1+|k-\ell|\leq(1+|k-s|)(1+|s-\ell|)$
for any
$k,s,\ell\in\zz^n$, we obtain,
for any given $A\in(0,1]$ and
for any  $j\in\zz$,
$k\in\zz^n$,
$\vec{f}\in[\cs_\infty'(\rn)]^m$,
and $x\in Q_{jk}$,
\begin{align}\label{pg}
&\sup_{y\in\rn}
\frac{|A_{Q_{jk}}(\varphi_j\ast\vec{f})(y)|
^A}{(1+2^j|x-y|)^{aA}}\\
&\quad=\sup_{s\in\zz^n}
\sup_{y\in Q_{js}}
\frac{|A_{Q_{jk}}(\varphi_j\ast\vec{f})(y)|^A}
{(1+2^j|x-y|)^{aA}}
\noz\\
&\quad\leq\sup_{s\in\zz^n}
\sup_{y\in Q_{js}}
\frac{\|A_{Q_{jk}}A_{Q_{js}}^{-1}\|^{A}
|A_{Q_{js}}(\varphi_j\ast\vec{f})(y)|^A}
{(1+2^j|x-y|)^{aA}}
\noz\\
&\quad\lesssim\sup_{s\in\zz^n}
\sup_{y\in Q_{js}}
\frac{(1+|k-s|)^{Ar}
|A_{Q_{js}}(\varphi_j\ast\vec{f})(y)|^A}
{(1+2^j|x-y|)^{aA}}
\noz\\
&\quad\sim\sup_{s\in\zz^n}
\lf(1+|k-s|\r)^{Ar}
\lf(1+|k-s|\r)^{-aA}\sup_{y\in Q_{js}}
\lf|A_{Q_{js}}(\varphi_j\ast\vec{f})(y)\r|^A\noz\\
&\quad\lesssim\sup_{s\in\zz^n}
\sum_{\ell\in\zz^n}\lf(1+|k-s|\r)^{-A(a-r)}
\lf(1+|\ell-s|\r)^{-A(R-r)}2^{jn}
\int_{Q_{j\ell}}
\lf|A_{Q_{j\ell}}
\lf(\varphi_j\ast\vec{f}\r)(z)\r|^A\,dz
\noz\\
&\quad\lesssim\sum_{\ell\in\zz^n}
\lf(1+|k-\ell|\r)^{-A(a-r)}2^{jn}
\int_{Q_{j\ell}}\lf|A_{Q_{j\ell}}
\lf(\varphi_j\ast\vec{f}\r)(z)\r|^A\,dz,\noz
\end{align}
where $r:=\frac{\beta}{p}$ and,
in the last step,
we used the fact that
$(1+|k-s|)(1+|\ell-s|)\geq(1+|k-\ell|)$
for any $k,\ell,s\in\zz^n$ and the fact that $a\in[r,\infty)$,
and chose an $R\in[a,\infty)$.
Let $A\in(0,1]$ satisfy $q/A>1$.
Using this, \eqref{pg}, and
the disjointness of the cubes $Q_{jk}$ for  any $k\in\zz^n$,
we further find that
\begin{align*}
&\sum_{Q\in\mathcal{D}_j}
\lf[2^{j\alpha}\sup_{y\in\rn}
\frac{|A_{Q}\lf(\varphi_j\ast\vec{f}\r)(y)|}
{(1+2^j|\cdot-y|)^{a}}
\mathbf{1}_{Q}(\cdot)\r]^q\\
&\quad=\sum_{k\in\zz^n}
2^{j\alpha q}
\lf[\sup _{y\in\rn}
\frac{|A_{Q_{jk}}\lf(\varphi_j\ast\vec{f}\r)(y)|^A}
{(1+2^j|\cdot-y|)^{aA}}
\mathbf{1}_{Q_{jk}}(\cdot)\r]^{q/A}\\
&\quad\lesssim
\lf[\sum_{k\in\zz^n}
\sum_{\ell\in\zz^n}\lf(1+|k-\ell|\r)^{-A(a-r)}
2^{jn}
\int_{Q_{j\ell}}
\lf|2^{j\alpha}A_{Q_{j\ell}}
\lf(\varphi_j\ast\vec{f}\r)(z)\r|^A\,dz
\mathbf{1}_{Q_{jk}}(\cdot)\r]^{q/A}.
\end{align*}
From
$a\in(\frac{n}{\min\{1,p,q\}}+r,\infty)$,
it follows that $\min\{1,p,q\}(a-r)>n$ and hence we can choose an $A\in(0,1]$
such that $A(a-r)>n$,
$p/A>1$,
and $q/A>1$.
Thus, by Lemma \ref{jcf}
and the Fefferman--Stein
vector-valued maximal inequality,
we conclude that,
for any $\vec{f}\in[\cs_\infty'(\rn)]^m$,
\begin{align*}
&\lf\|\lf\{\sum_{j\in\zz}
\sum_{Q\in\mathcal{D}_j}\lf[2^{j\alpha}
\sup_{y\in\rn}
\frac{|A_{Q}(\varphi_j\ast\vec{f})(y)|}
{(1+2^j|\cdot-y|)^{a}}
\mathbf{1}_{Q}\r]^q\r\}^{1/q}\r\|_{L^p(\rn)}\\
&\quad\lesssim
\lf\|\lf\{\sum_{j\in\zz}
\lf[\mathcal{M}
\lf(\sum_{Q\in{\mathcal{D}_j}}\lf[2^{j\alpha}
\lf|A_Q\lf(\varphi_j\ast\vec{f}\r)\r|
\mathbf{1}_{Q}\r]^A\r)\r]^{q/A}\r\}^{A/q}\r\|_
{L^{p/A}(\rn)}^{1/A}\\
&\quad\lesssim
\lf\|\vec f\r\|_
{\dot{F}^{\alpha,q}_{p}(\{A_Q\})},
\end{align*}
which implies that
\eqref{21} holds true for any
$\vec{f}\in[\cs_\infty'(\rn)]^m$.
Now, for any
$\vec{f}\in [\cs_\infty'(\rn)]^m$,
let
\begin{equation}\label{38x}
\lf\|\vec{f}\r\|_
{\dot{F}^{\alpha,q}_{p}(\{A_Q\})}^{\star\star}:=
\lf\|\lf[\sum_{j\in\zz}
\sum_{Q\in\mathcal{D}_j}
2^{j\alpha q}
\sup_{z\in Q}
\sup_{y\in\rn}
\frac{|A_Q(\varphi_j\ast\vec{f})(y)|^q}
{(1+2^j|z-y|)^{aq}}
\mathbf{1}_Q\r]^{1/q}\r\|_{L^p(\rn)}.
\end{equation}
From \eqref{21} and Lemma \ref{eq}, we infer that,
for any $\vec{f}\in[\cs'_\infty(\rn)]^m$,
$$
\lf\|\vec{f}\r\|_
{\dot{F}^{\alpha,q}_{p}(\{A_Q\})}^{\star}
\lesssim
\lf\|\vec{f}\r\|_
{\dot{F}^{\alpha,q}_{p}(\{A_Q\})}
\sim
\lf\|\vec{f}\r\|_
{\dot{F}^{\alpha,q}_{p}(W)}.
$$
By this,
to complete the proof of \eqref{r45},
we still need  to prove that,
for any
$\vec{f}\in[\cs'_\infty(\rn)]^m$,
\begin{equation}\label{572}
\lf\|\vec{f}\r\|
_{\dot{F}^{\alpha,q}_{p}(W)}^{\star}
\lesssim
\lf\|\vec{f}\r\|_
{\dot{F}^{\alpha,q}_{p}(\{A_Q\})}^{\star\star}
\lesssim
\lf\|\vec{f}\r\|_{\dot{F}^{\alpha,q}_{p}(\{A_Q\})}^{\star}.
\end{equation}
We first show the
first inequality of \eqref{572}.
For any $j\in\zz$ and $x\in\rn$,
let
$$
h_j(x):=
2^{j\alpha}
\sup_{y\in\rn}
\frac{|W^{1/p}(x)(\varphi_j\ast\vec{f})(y)|}
{(1+2^j|x-y|)^a},
$$
$$
k_j(x):=
\sum_{Q\in\mathcal{D}_j}
\lf|Q\r|^{-\alpha/n}\sup_{z\in Q}
\sup_{y\in\rn}
\frac{|A_{Q}(\varphi_j\ast\vec{f})(y)|}
{(1+2^j|z-y|)^a}
\mathbf{1}_Q(x),
$$
and
$$
\gamma_j(x):=
\sum_{Q\in\mathcal{D}_j}
\lf\|W^{1/p}(x)A_Q^{-1}\r\|
\mathbf{1}_Q(x).
$$
It is obvious that,
for any $j\in\zz$ and $x\in\rn$,
\begin{align}\label{39x}
h_j(x)&=\sum_{Q\in\mathcal{D}_j}2^{j\alpha}
\sup_{y\in\rn}
\frac{|W^{1/p}(x)A_Q^{-1}A_Q(\varphi_j\ast\vec{f})(y)|}
{(1+2^j|x-y|)^a}\mathbf{1}_{Q}(x)\\
&\leq\sum_{Q\in\mathcal{D}_j}
2^{j\alpha}\lf\|W^{1/p}(x)A_Q^{-1}\r\|
\sup_{y\in\rn}
\frac{|A_{Q}(\varphi_j\ast\vec{f})(y)|}
{(1+2^j|x-y|)^a}
\mathbf{1}_{Q}(x)\leq \gamma_j(x)k_j(x).\noz
\end{align}
Notice that $k_j$ is a constant
on any given cube $Q\in\mathcal{D}_j$,
which implies that
\begin{equation}\label{this}
E_j(k_j)=k_j,
\end{equation}
where $E_j$ is the same as in  Lemma \ref{c38}.
Then, by \eqref{pjida}, \eqref{39x}, Lemma \ref{c38},
\eqref{this}, and \eqref{38x},
we have,
for any $\vec{f}\in[\cs'_\infty(\rn)]^m$,

\begin{align}\label{57}
\lf\|\vec{f}\r\|_
{\dot{F}^{\alpha,q}_{p}(W)}^{\star}
&=\lf\|\lf\{h_{j}\r\}_{j\in\zz}\r\|_
{L^p(\ell^q)}
\leq
\lf\|\lf\{\gamma_jE_j\lf(k_j\r)\r\}_{j\in\zz}\r\|
_{L^p(\ell^q)}\\
&\lesssim
\lf\|\lf\{E_j\lf(k_j\r)\r\}_{j\in\zz}\r\|
_{L^p(\ell^q)}
\sim\lf\|\lf\{k_j\r\}_{j\in\zz}\r\|
_{L^p(\ell^q)}
\sim\lf\|\vec{f}\r\|
_{\dot{F}^{\alpha,q}_{p}(\{A_Q\})}^{\star\star},\noz
\end{align}
which is just the first inequality of \eqref{572}.
Next, we prove the
second inequality of \eqref{572}.
Indeed, using a geometrical observation, we find that
$1+2^j|x-y|\sim 1+|s-k|\sim 1+2^j|z-y|$
for any $x,z\in Q_{jk}$ and $y\in Q_{js}$.
From this, we deduce that,
for any $a\in(0,\infty)$,
$j\in\zz$,
$k\in\zz^n$,
and $x\in Q_{jk}$,
\begin{align*}
\sup_{z\in Q_{jk}}
\sup_{y\in\rn}
\frac{|A_{Q_{jk}}(\varphi_j\ast\vec{f})(y)|}
{(1+2^j|z-y|)^{a}}
\sim\sup_{y\in\rn}
\frac{|A_{Q_{jk}}(\varphi_j\ast\vec{f})(y)|}
{(1+2^j|x-y|)^{a}},
\end{align*}
which implies that,
for any $\vec{f}\in[\cs'_\infty(\rn)]^m$,
\begin{equation}\label{72}
\lf\|\vec{f}\r\|_
{\dot{F}^{\alpha,q}_{p}(\{A_Q\})}^{\star\star}\sim\lf\|\vec{f}\r\|
_{\dot{F}^{\alpha,q}_{p}(\{A_Q\})}^{\star}.
\end{equation}
Thus, both \eqref{57} and \eqref{72} imply \eqref{572}, and hence
\eqref{pmfd} holds true for any
$\vec{f}\in[\cs'_{\infty}(\rn)]^m$,
which completes the proof of Theorem \ref{pmf}.
\end{proof}

We now establish the Lusin-area function
characterization of matrix-weighted
Triebel--Lizorkin spaces.

\begin{theorem}\label{lusin}
Let $\alpha\in\mathbb{R}$,
$p\in(0,\infty)$,
$q\in(0,\infty]$,
and $W\in A_p(\rn,\mathbb{C}^m)$.
Assume that $\varphi\in\cs(\rn)$
satisfies both (T1) and (T2) of Definition \ref{FBW}.
Then $\vec{f}\in\dot{F}^{\alpha,q}_{p}(W)$
if and only if
$\vec{f}\in[\cs_\infty'(\rn)]^m$ with
 $\|\vec{f}\|_{\dot{F}^{\alpha,q}_{p}(W)}^{\square}<\infty$,
where
\begin{equation*}
\lf\|\vec f\r\|^{\square}
_{\dot{F}^{\alpha,q}_{p}(W)}
:=\lf\|\lf[\sum_{j\in\zz} 2^{j\alpha q}
\fint_{B(\cdot, 2^{-j})}
\lf|W^{1/p}(\cdot)
\lf(\varphi_j\ast\vec{f}\r)(y)\r|^q
\,dy\r]^{1/q}\r\|_{L^p(\rn)}
\end{equation*}
with usual modification made when $q=\infty$.
 Moreover, for any $\vec{f}\in[\cs_{\infty}'(\rn)]^m$,
\begin{equation}\label{lusind}
\lf\|\vec f\r\|_{\dot{F}^{\alpha,q}_{p}(W)}\sim\lf\|\vec f\r\|^{\square}
_{\dot{F}^{\alpha,q}_{p}(W)},
\end{equation}
where the positive equivalence constants are
 independent of $\vec{f}$.
\end{theorem}
\begin{proof}
Let all the symbols be the same as in  the present theorem.
We now claim that,
for any $\vec{f}\in[\cs'_\infty(\rn)]^m$,
\begin{equation*}
\lf\|\vec{f}\r\|
_{\dot{F}^{\alpha,q}_{p}(W)}^{\star}
\sim\lf\|\vec{f}\r\|
_{\dot{F}^{\alpha,q}_{p}(W)}^{\square},
\end{equation*}
if $a$ is sufficiently large.
Then, by Theorem \ref{pmf},
we conclude that the present theorem holds true.

First, we prove that, when $a\in(0,\infty)$, then,
for any $\vec{f}\in[\cs'_{\infty}(\rn)]^m$,
\begin{equation}\label{zb}
\lf\|\vec{f}\r\|
_{\dot{F}^{\alpha,q}_{p}(W)}^{\square}
\lesssim
\lf\|\vec{f}\r\|
_{\dot{F}^{\alpha,q}_{p}(W)}^{\star}.
\end{equation}
By the change of variables, the fact that $1+2^j|y|\sim 1$
for any
$y\in B(\mathbf{0},2^{-j})$, and \eqref{pjd},
 we conclude that,
for any given $q\in(0,\infty)$
and $a\in(0,\infty)$,
and for any $j\in\zz$ and $x\in\rn$,
\begin{align*}
&\fint_{B(x, 2^{-j})}
\lf|W^{1/p}(x)\lf(\varphi_j\ast\vec{f}\r)(y)\r|^q\,dy\\
&\quad=\fint_{B(\mathbf{0}, 2^{-j})}
\lf|W^{1/p}(x)\lf(\varphi_j\ast\vec{f}\r)(x+y)\r|^q\,dy
\lesssim
\sup_{y\in B(\mathbf{0}, 2^{-j})}\lf|W^{1/p}(x)
\lf(\varphi_j\ast\vec{f}\r)(x+y)\r|^q\\
&\quad\sim\sup_{y\in B(\mathbf{0}, 2^{-j})}
\frac{|W^{1/p}(x)(\varphi_j\ast\vec{f})(x+y)|^q}
{(1+2^j|y|)^{aq}}\lesssim\lf[\lf(\varphi_j^*
\vec{f}\r)_a^{(W,p)}(x)\r]^q,
\end{align*}
which implies that \eqref{zb} holds true.
		
Next, we show that,
for any $\vec{f}\in[\cs_\infty'(\rn)]^m$,
\begin{equation}\label{231}
\lf\|\vec{f}\r\|
_{\dot{F}^{\alpha,q}_{p}(W)}^{\star}
\lesssim
\lf\|\vec{f}\r\|
_{\dot{F}^{\alpha,q}_{p}(W)}^{\square},
\end{equation}
if $a$ is sufficiently large.
Using \eqref{572},  to prove \eqref{231},
we only need to show that,
for any $\vec{f}\in[\cs_\infty'(\rn)]^m$,
\begin{equation}\label{23}
\lf\|\vec{f}\r\|
_{\dot{F}^{\alpha,q}_{p}(\{A_Q\})}^{\star}
\lesssim
\lf\|\vec{f}\r\|
_{\dot{F}^{\alpha,q}_{p}(W)}^{\square},
\end{equation}
if $a$ is sufficiently large.
For any given $A\in(0,1]$
satisfying that $q/A>1$ and $p/A>1$,
we choose an $a\in(0,\infty)$
sufficiently large
such that $A(a-r)>n$, where $r:=\frac{\beta}{p}$ and
$\beta$ is the doubling exponent of $W$.
Then, by \eqref{pg},  the change of variables,
the fact that, for any
$z\in Q_{j\ell}$ and $s\in B(\mathbf{0},2^{-j})$,
$z-s\in Q_{j(t+\ell)}$ for
some $t:=(t_1,\ldots,t_n)\in\zz^n$ satisfying
$|t|_{\infty}:=\max\{t_d:\ d\in\{1,\ldots,n\}\}\leq 1$,
and
Lemma \ref{wsz},
we conclude that,
for any $j\in\zz$,
$k\in\zz^n$, and $x\in Q_{jk}$,
\begin{align}\label{g2}
&\sup_{y\in\rn}
\frac{|A_{Q_{jk}}(\varphi_j\ast\vec{f})(y)|^A}
{(1+2^j|x-y|)^{aA}}
\\
&\quad\sim
\sum_{\ell\in\zz^n}
\lf(1+|k-\ell|\r)^{-A(a-r)}2^{jn}
\fint_{B(\mathbf{0}, 2^{-j})}
\int_{Q_{j\ell}}
\lf|A_{Q_{j\ell}}
\lf(\varphi_j\ast\vec{f}\r)(z)\r|^A\,dz\,ds
\noz\\
&\quad\lesssim
\sum_{\ell\in\zz^n}
\lf(1+|k-\ell|\r)^{-A(a-r)}2^{jn}\noz\\
&\qquad\times
\sum_{\{t\in\zz^n:\,|t|_{\infty}\leq 1\}}
\fint_{B(\mathbf{0}, 2^{-j})}\int_{Q_{j(\ell+t)}}
\lf|A_{Q_{j\ell}}
\lf(\varphi_j\ast\vec{f}\r)(s+z)\r|^A\,dz\,ds
\noz\\
&\quad\lesssim
\sum_{\ell\in\zz^n}
\lf(1+|k-\ell|\r)^{-A(a-r)}2^{jn}\noz\\
&\qquad\times\sum_{\{t\in\zz^n:\,|t|_{\infty}\leq 1\}}
\int_{Q_{j(\ell+t)}}
\fint_{B(\mathbf{0}, 2^{-j})}
\lf|A_{Q_{j\ell}}
\lf(\varphi_j\ast\vec{f}\r)(s+z)\r|^A\,ds\,dz
\noz\\
&\quad\lesssim
\sum_{\ell\in\zz^n}
\lf(1+|k-\ell|\r)^{-A(a-r)}2^{jn}\noz\\
&\qquad\times\sum_{\{t\in\zz^n:\,|t|_{\infty}\leq 1\}}
\int_{Q_{j(\ell+t)}}
\fint_{B(\mathbf{0}, 2^{-j})}
\lf|A_{Q_{j(\ell+t)}}
\lf(\varphi_j\ast\vec{f}\r)(s+z)\r|^A\,ds\,dz
\noz.
\end{align}

Now, we prove \eqref{23} by considering two cases on $p$.

Case 1) $p\in(0,1]$. In this case,
noticing that
$\mathbf{1}_{Q_{j(\ell+t)}}
=\sum_{Q\in\mathcal{D}_j}
(\mathbf{1}_Q\mathbf{1}_{Q_{j(\ell+t)}})$,
we then have
\begin{align}\label{gjr}
&\int_{Q_{j(\ell+t)}}2^{j\alpha}
\fint_{B(\mathbf{0}, 2^{-j})}
\lf|A_{Q_{j(\ell+t)}}
\lf(\varphi_j\ast\vec{f}\r)(s+z)\r|^A\,ds\,dz\\
&\quad=\int_{Q_{j(\ell+t)}}
\sum\limits_{Q\in\mathcal{D}_j}2^{j\alpha}
\fint_{B(\mathbf{0}, 2^{-j})}
\lf|A_Q\varphi_j\ast\vec{f}(s+z)\r|^A\,ds
\mathbf{1}_Q(z)\,dz\noz\\
&\quad
=\int_{Q_{j(\ell+t)}}g_j(z)\,dz,\noz
\end{align}
where, for any $z\in\rn$,
$$g_j(z):=\sum\limits_{Q\in\mathcal{D}_j}2^{j\alpha}
\fint_{B(\mathbf{0}, 2^{-j})}
\lf|A_Q\varphi_j\ast\vec{f}(s+z)\r|^A\,ds
\mathbf{1}_Q(z).$$
For any given $x\in Q_{jk}$, let
$B_x:=B(x_{k,\ell,t},r_{k,\ell,t})$
be the smallest ball containing both $x$ and
the dyadic cube $Q_{j(\ell+t)}$.
Then $r_{k,\ell,t} \sim2^{-j}(1+|k-\ell-t|)$.
Since $|t|_{\infty}\leq1$,
it follows that
\begin{equation}\label{rgx}
r_{k,\ell,t}\sim2^{-j}(1+|k-\ell|).
\end{equation}
Using this and \eqref{gjr},
we obtain, for any $x\in Q_{jk}$,
\begin{align}\label{g1}
&\int_{Q_{j(\ell+t)}}2^{j\alpha}
\fint_{B(\mathbf{0}, 2^{-j})}
\lf|A_{Q_{j(\ell+t)}}\lf(\varphi_j\ast\vec{f}\r)(s+z)
\r|^A\,ds\,dz\\
&\quad\leq
\int_{B_x}g_j(z)\,dz
\lesssim
2^{-jn}(1+|k-\ell|)^n\mathcal{M}(g_j)(x)\noz.
\end{align}
By both \eqref{g2} and \eqref{g1},
we conclude that,
for any $x\in\rn$,
\begin{align}\label{gz2}
&2^{j\alpha}\sum_{Q\in \mathcal{D}_j}
\sup_{y\in\rn}
\frac{|A_{Q}(\varphi_j\ast\vec{f})(y)|^A}
{(1+2^j|x-y|)^{aA}}
\mathbf{1}_{Q}(x)\\
&\quad\lesssim\sum_{\ell\in\zz^n}
(1+|k-\ell|)^{-A(a-r)+n}
\mathcal{M}(g_j)(x)
\lesssim\mathcal{M}(g_j)(x)\noz,
\end{align}
where,
in the last step,
we used the assumption  $A(a-r)>2n$.
From \eqref{gz2},  we further deduce that,
for any $x\in\rn$,
\begin{equation}\label{gz}
\sum_{Q\in \mathcal{D}_j}
\lf[2^{j\alpha}
\sup_{y\in\rn}
\frac{|A_{Q}(\varphi_j\ast\vec{f})(y)|}
{(1+2^j|x-y|)^{a}}
\mathbf{1}_{Q}(x)\r]^q
\quad\lesssim
\lf[\mathcal{M}(g_j)(x)\r]^{q/A}.
\end{equation}
By \eqref{gz},
the Fefferman--Stein
vector-valued maximal inequality together
with $p/A>1$ and $q/A>1$,
the H\"older inequality, and Lemma \ref{ainf},
we find that,
for any $\vec{f}\in[\cs'_\infty(\rn)]^m$,
\begin{align*}
\lf\|\vec{f}\r\|
_{\dot{F}^{\alpha,q}_{p}(\{A_Q\})}^\star
&\lesssim
\lf\|\lf\{\sum_{j\in\zz}
\lf[\mathcal{M}(g_j)\r]^{q/A}\r\}^{A/q}\r\|
_{L^{p/A}(\rn)}^{1/A}\\
&\lesssim\lf\|\lf\{\sum_{j\in\zz}
\sum\limits_{Q\in\mathcal{D}_j}2^{j\alpha q}
\lf[\fint_{B(\mathbf{0}, 2^{-j})}
\lf|A_Q\varphi_j\ast\vec{f}(\cdot+z)\r|^A\,dz\r]^{q/A}
\mathbf{1}_Q\r\}^{1/q}\r\|_{L^p(\rn)}\\
&\lesssim
\lf\|\lf[\sum_{j\in\zz}
\sum\limits_{Q\in\mathcal{D}_j}
2^{j\alpha q}
\fint_{B(\mathbf{0}, 2^{-j})}
\lf|A_Q\varphi_j\ast\vec{f}(\cdot+z)\r|^q\,dz
\mathbf{1}_Q\r]^{1/q}\r\|_{L^p(\rn)}\\
&\lesssim
\lf\|\lf[\sum_{j\in\zz}
\sum\limits_{Q\in\mathcal{D}_j}2^{j\alpha q}
\fint_{B(\mathbf{0}, 2^{-j})}
\lf\|A_QW^{-1/p}(\cdot)\r\|^q
\lf|W^{1/p}(\cdot)\varphi_j\ast\vec{f}(\cdot+z)\r|^q\,dz
\mathbf{1}_Q\r]^{1/q}\r\|_{L^p(\rn)}\\
&\lesssim
\lf\|\vec{f}\r\|
_{\dot{F}^{\alpha,q}_{p}(W)}^{\square}.
\end{align*}
Thus, \eqref{23} holds true when $p\in(0,1]$.
	
Case 2) $p\in(1,\infty)$. In this case,
from \eqref{g2},
the H\"older inequality,
and  Lemma \ref{a123},
we infer that,
for any $x\in Q_{jk}$,
\begin{align}\label{g3}
&\sup_{y\in\rn}
\frac{|A_{Q_{jk}}(\varphi_j\ast\vec{f})(y)|^A}
{(1+2^j|x-y|)^{aA}}\\
&\quad\lesssim
\sum_{\ell\in\zz^n}
\lf(1+|k-\ell|\r)^{-A(a-r)}2^{jn}
\sum_{\{t\in\zz^n:\,|t|_{\infty}\leq 1\}}
\Bigg[\int_{Q_{j(\ell+t)}}
\lf\|A_{Q_{j(\ell+t)}}W^{-1/p}(z)\r\|^A\noz\\
&\qquad\times\fint_{B(\mathbf{0}, 2^{-j})}
\lf|W^{1/p}(z)
\lf(\varphi_j\ast\vec{f}\r)(s+z)\r|^A\,ds\,dz
\Bigg]\noz\\
&\quad\lesssim
\sum_{\ell\in\zz^n}
\lf(1+|k-\ell|\r)^{-A(a-r)}2^{jn}
\sum_{\{t\in\zz^n:\,|t|_{\infty}\leq 1\}}
\left\{\lf[\int_{Q_{j(\ell+t)}}
\lf\|A_{Q_{j(\ell+t)}}
W^{-1/p}(z)\r\|^{p'}\,dz\r]^{A/p'}\right.\noz\\
&\qquad\left.\times\lf[
\int_{Q_{j(\ell+t)}}\lf\{
\fint_{B(\mathbf{0}, 2^{-j})}
\lf|W^{1/p}(z)\lf(\varphi_j
\ast\vec{f}\r)(s+z)\r|^A\,ds\r\}
^{\frac{p'}{p'-A}}\,dz\r]
^{\frac{p'-A}{p'}}\right\}\noz\\
&\quad\lesssim
\sum_{\ell\in\zz^n}
\lf(1+|k-\ell|\r)
^{-A(a-r)}
2^{jn(1-\frac{A}{p'})}\noz\\
&\qquad\times
\sum_{\{t\in\zz^n:\,|t|_
{\infty}\leq 1\}}\lf\{\int_{Q_{j(\ell+t)}}
\lf[\fint_{B(\mathbf{0}, 2^{-j})}
\lf|W^{1/p}(z)\lf(\varphi_j\ast
\vec{f}\r)(s+z)\r|^A\,ds\r]
^{\frac{p'}{p'-A}}\,dz\r\}
^{\frac{p'-A}{p'}}\noz.
\end{align}
For any given $x\in\rn$,
let $B_x:=B(x_{k,\ell,t},r_{k,\ell,t})$ be the same as in
Case 1). Notice that,
for any $M>n$, $$\sup_{k\in\zz^n}\sum_{\ell\in\zz^n}
(1+|k-\ell|)^{-M}=\sup_{k\in\zz^n}
\sum_{k-\ell\in\zz^n}(1+|k-\ell|)^{-M}=
\sum_{\ell\in\zz^n}(1+|\ell|)^{-M}\lesssim1.$$
By this, \eqref{g3},  \eqref{rgx},
the H\"older inequality,
and the disjointness of $Q_{jk}$ for any $k\in\zz^n$,
we conclude that, for any $x\in\rn$,
\begin{align}\label{g4}
&\sum_{k\in\zz^n}
\sup_{y\in\rn}
\frac{|A_{Q_{jk}}(\varphi_j\ast\vec{f})(y)|^q}
{(1+2^j|x-y|)^{aq}}
\mathbf{1}_{Q_{jk}}(x)\\
&\quad\lesssim
\lf[\sum_{k\in\zz^n}
\sum_{\ell\in\zz^n}
\lf(1+|k-\ell|\r)^{-A(a-r)}2^{jn(1-\frac{A}{p'})}\noz\r.\\
&\qquad\lf.\times
\lf\{\int_{B_x}
\lf[\fint_{B(\mathbf{0}, 2^{-j})}
\lf|W^{1/p}(z)\lf(\varphi_j\ast\vec{f}\r)(s+z)\r|^A\,ds\r]
^{\frac{p'}{p'-A}}\,dz\r\}
^{\frac{p'-A}{p'}}
\mathbf{1}_{Q_{jk}}(x)\r]^{q/A}\noz\\
&\quad\lesssim
\lf\{\sum_{k\in\zz^n}
\sum_{\ell\in\zz^n}
\lf(1+|k-\ell|\r)^{-A(a-r)
+\frac{(p'-A)n}{p'}}\r.\noz\\
&\qquad\lf.\times
\lf[\mathcal{M}\lf(\lf[\fint_{B(\mathbf{0}, 2^{-j})}
\lf|W^{1/p}(\cdot)
\lf(\varphi_j\ast\vec{f}\r)(\cdot+z)\r|^A\,dz\r]
^{\frac{p'}{p'-A}}\r)(x)\r]
^{\frac{p'-A}{p'}}\mathbf{1}_{Q_{jk}}(x)\r\}^{q/A}\noz\\
&\quad\lesssim\lf\{\mathcal{M}
\lf(\lf[\fint_{B(\mathbf{0}, 2^{-j})}
\lf|W^{1/p}(\cdot)
\lf(\varphi_j\ast\vec{f}\r)(\cdot+z)\r|^A\,dz\r]
^{\frac{p'}{p'-A}}\r)(x)\r\}^{\frac{(p'-A)q}{Ap'}}\noz\\
&\qquad\times\lf[ \sup_{k\in\zz^n}\sum_{\ell\in\zz^n}(1+|k-\ell|)
^{-A(a-r)+\frac{(p'-A)n}{p'}}\r]^{q/A}\noz\\
&\quad\lesssim
\lf\{\mathcal{M}
\lf(\lf[\fint_{B(\mathbf{0}, 2^{-j})}
\lf|W^{1/p}(\cdot)
\lf(\varphi_j\ast\vec{f}\r)(\cdot+z)\r|^A\,dz\r]
^{\frac{p'}{p'-A}}\r)(x)\r\}^{\frac{(p'-A)q}{Ap'}},\noz
\end{align}
where, in the last step,
we chose a sufficiently large $a\in(0,\infty)$
such that $A(a-r)-\frac{(p'-A)n}{p'}>n$.
Noticing that
$\frac{p(p'-A)}{Ap'}=\frac{A+(1-A)p}{A}>1$,
choose
$A\in(0,1)$ sufficiently small,
and hence $\frac{(p'-A)q}{Ap'}>1$
and $q/A>1$.
From this, \eqref{g4},  the Fefferman--Stein
vector-valued maximal inequality,
and the H\"older inequality,
we deduce that,
for any $\vec{f}\in[\cs'_\infty(\rn)]^m$,
\begin{align*}
&\lf\|\vec{f}\r\|
_{\dot{F}^{\alpha,q}_{p}(\{A_Q\})}^\star\\
&\quad\lesssim
\lf\|\lf\{\sum_{j\in\zz}2^{j\alpha q}
\lf[\mathcal{M}\lf(\lf[
\fint_{B(\mathbf{0}, 2^{-j})}
\lf|W^{1/p}(\cdot)
\lf(\varphi_j\ast\vec{f}\r)(\cdot+z)\r|^A\,dz\r]
^{\frac{p'}{p'-A}}\r)\r]
^{\frac{(p'-A)q}{Ap'}}\r\}
^{\frac{1}{q}}\r\|_{L^p(\rn)}\\
&\quad\sim
\lf\|\lf\{\sum_{j\in\zz}
\lf[\mathcal{M}
\lf(\lf[\fint_{B(\mathbf{0}, 2^{-j})}
\lf|2^{j\alpha}W^{1/p}(\cdot)
\lf(\varphi_j\ast\vec{f}\r)(\cdot+z)\r|^A\,dz\r]
^{\frac{p'}{p'-A}}\r)\r]
^{\frac{(p'-A)q}{Ap'}}\r\}
^{\frac{Ap'}{(p'-A)q}}\r\|
_{L^{\frac{p(p'-A)}{p'A}}(\rn)}\\
&\quad\lesssim
\lf\|\lf\{\sum_{j\in\zz}
2^{j\alpha q}
\lf[\fint_{B(\mathbf{0}, 2^{-j})}
\lf|W^{1/p}(x)
\lf(\varphi_j\ast\vec{f}\r)(\cdot+z)\r|^A\,dz\r]
^{q/A}\r\}^{1/q}\r\|_{L^p(\rn)}\\
&\quad\lesssim
\lf\|\lf\{\sum_{j\in\zz}
2^{j\alpha q}
\lf[\fint_{B(\mathbf{0}, 2^{-j})}
\lf|W^{1/p}(x)
\lf(\varphi_j\ast\vec{f}\r)(\cdot+z)\r|^q
\,dz\r]\r\}^{1/q}\r\|_{L^p(\rn)}\\
&\quad\sim
\lf\|\vec{f}\r\|
_{\dot{F}^{\alpha,q}_{p}(W)}^\square.
\end{align*}
Thus, \eqref{23} holds true when $p\in(0,\infty)$.

Combining both Cases 1) and 2), we conclude that \eqref{23}
holds true.
From \eqref{231},  \eqref{572},
and Theorem \ref{pmf}, we infer that \eqref{lusind}
holds true for any
$\vec{f}\in[\cs'_{\infty}(\rn)]^m$, which then completes
the proof of Theorem \ref{lusin}.
\end{proof}

\begin{remark}\label{lr}
Theorem
\ref{lusin} when $m=1$ and $W=1$ is just
\cite[Theorem 2.12.1]{T1}
which is the Lusin-area function
characterization of
Triebel--Lizorkin spaces.
\end{remark}

In what follows,
we establish the $g_\lambda^*-$
function characterization of
$\dot{F}^{\alpha,q}_{p}(W)$.
First, we give the following technical lemma.

\begin{lemma}\label{a.e.}
Let $\varphi\in\cs(\rn)$
with $\supp\widehat{\varphi}$
being bounded and
away from the origin.
Then,
for any $f\in\cs_\infty'(\rn)$,
$\varphi\ast f\in C^{\infty}(\rn)\cap\cs'(\rn)$
and
$\supp(\varphi\ast f)^{\land}
\subseteq\supp\widehat{\varphi}.$
\end{lemma}

\begin{proof}
Since $\overline{\supp\widehat{\varphi}}$
is bounded and away from the origin,
then we deduce that,
for any
$\psi\in\cs(\rn)$, $\overline{\supp(\widehat{\varphi}\widehat{\psi})}
\subseteq\overline{\supp\widehat{\varphi}}$
is bounded and away from the origin,
which implies that
\begin{equation}\label{sw}
\varphi\ast\psi\in\cs_\infty(\rn).
\end{equation}
By an argument similar to that used in the proof
 of \cite[Proposition 2.3.4(b)]{G1},
we find that, if $f\in\cs_{\infty}'(\rn)$, then
there exists a positive constant $C$ and
$k,\ell\in\zz_+$ such that, for any $\phi\in\cs_{\infty}(\rn)$,
$$\lf|\lf\langle f,\phi\r\rangle\r|\leq C \sum_{\{\mu,\,\nu\in\zz_+^n:
\,|\mu|\leq k,\,|\nu|\leq\ell\}}\rho_{\mu,\nu}(\phi),$$
where
\begin{equation}\label{rhocs}
\rho_{\mu,\nu}(\phi):=
\sup\limits_{x\in\rn}|x^\mu\partial^\nu\phi(x)|.
\end{equation}
From this,  $f\in\cs'_{\infty}(\rn)$, and \eqref{sw},
we infer that
there exist $k,\ell\in\zz_+$ such that, for any $\psi\in\cs(\rn)$,
\begin{align*}
\lf|\langle \varphi\ast f, \psi\rangle\r|
&=\lf|\langle f,\widetilde{\varphi}\ast\psi\rangle\r|\\
&\lesssim\sum_{\{\mu,\,\nu\in\zz_+^n:\,|\mu|\leq k,|\nu|\leq \ell\}}
\rho_{\mu,\nu}(\widetilde{\varphi}\ast\psi)\\
&\lesssim\sum_{\{\mu,\,\nu\in\zz_+^n:\,|\mu|\leq k,|\nu|\leq \ell\}}
\sup_{x\in\rn}
\int_\rn\sum_{|\mu'|\leq|\mu|}
\lf|x-y\r|^{|\mu'|}|y|^{|\mu|-|\mu'|}
\lf|\partial_x^\nu\psi(x-y)\r||\widetilde{\varphi}(y)|\,dy\\
&\lesssim_{\varphi}\sum_{\{\mu,\,\nu\in\zz_+^n:\,|\mu|\leq k,|\nu|\leq
\ell\}}
\rho_{\mu,\nu}(\psi),
\end{align*}
where, in the last inequality, the implicit positive constant depends on $\varphi$,
which implies that
$\varphi\ast f\in\cs'(\rn)$.
By this, we conclude that,
for any $\gamma\in\cs(\rn)$
with $\supp{\gamma}
\subset(\rn\setminus\supp\widehat{\varphi})$,
$$
\lf\langle \lf(\varphi\ast f\r)^\land,\gamma\r\rangle
=\lf\langle \varphi\ast f,\widehat\gamma\r\rangle
=\lf\langle f, \widetilde{\varphi}\ast\widehat{\gamma}\r\rangle
=\lf\langle f,\widetilde{\lf(\widehat{\varphi}\gamma\r)^\vee}\r\rangle
=0,
$$
which implies that
$\supp(\varphi\ast f)^{\land}
\subseteq\supp\widehat{\varphi}.$
Then, from \cite[Theorem 2.3.21]{G1},
we deduce that
$\varphi\ast f\in L^1_{\loc}(\rn)$,
which completes the proof of Lemma \ref{a.e.}.
\end{proof}

\begin{theorem}\label{glambda}
Let  $\alpha\in\rr$,
$p\in(0,\infty)$, $q\in(0,\infty]$,
$W\in A_p(\rn,\mathbb{C}^m)$,
and $\lambda\in(\frac{1}{\min\{1,p,q\}}
+\frac{\beta}{np},\infty)$, where $\beta$ is the
doubling exponent of $W$.
Assume that $\varphi\in\cs(\rn)$
satisfies both (T1) and (T2) of Definition \ref{FBW}.
Then $\vec{f}\in\dot{F}^{\alpha,q}_{p}(W)$
if and only if
$\vec{f}\in[\cs'_\infty(\rn)]^m$
and $\|\vec{f}\|_{\dot{F}^{\alpha,q}_{p}(W)}
^{\clubsuit}<\infty$,
where
\begin{equation}\label{gl*}
\lf\|\vec{f} \r\|
_{\dot{F}^{\alpha,q}_{p}(W)}^\clubsuit
:=\lf\|\lf\{\sum_{j\in\zz}2^{j\alpha q}
2^{jn}\int_\rn
\frac{|W^{1/p}(\cdot)
(\varphi_j\ast\vec{f})(y)|^q}{(1+2^j|\cdot-y|)^{\lambda nq}}
\,dy\r\}
^{1/q}\r\|_{L^p(\rn)}
\end{equation}
with usual modification made when $q=\infty$.
Moreover, for any $\vec{f}\in[\cs_{\infty}'(\rn)]^m$,
\begin{equation}\label{glambdad}
\lf\|\vec f\r\|_{\dot{F}^{\alpha,q}_{p}(W)}\sim\lf\|\vec{f} \r\|
_{\dot{F}^{\alpha,q}_{p}(W)}^\clubsuit,
\end{equation}
where the positive equivalence constants are independent of $\vec{f}$.
\end{theorem}

\begin{proof}
Let all the symbols be the same as in  the present theorem.
First, we prove that, for any $\vec{f}\in[\cs'_{\infty}(\rn)]^m$,
\begin{equation}\label{wxg}
\lf\|\vec f\r\|
_{\dot{F}^{\alpha,q}_{p}(W)}
\lesssim\lf\|\vec f\r\|
_{\dot{F}^{\alpha,q}_{p}(W)}^\clubsuit.
\end{equation}
Indeed, by an observation that, for any  $x\in\rn$ and  $y\in B(x, 2^{-j})$,
$1+2^j|x-y|\sim 1$, we conclude that,
for any $x\in\rn$ and $\vec{f}\in[\cs'_{\infty}(\rn)]^m$,
\begin{align*}
&\fint_{B(x, 2^{-j})}
\lf|W^{1/p}(x)
\lf(\varphi_j\ast\vec{f}\r)(y)\r|^q\,dy\\
&\quad\sim 2^{jn}\int_{B(x,2^{-j})}\frac{|W^{1/p}(x)
(\varphi_j\ast\vec{f})(y)|^q}{(1+2^j|x-y|)^{\lambda nq}}
\,dy\lesssim 2^{jn}\int_{\rn}\frac{|W^{1/p}(x)
(\varphi_j\ast\vec{f})(y)|^q}{(1+2^j|x-y|)^{\lambda nq}}
\,dy,
\end{align*}
which implies that
$\|\vec f\|
_{\dot{F}^{\alpha,q}_{p}(W)}^\square
\lesssim\|\vec f\|
_{\dot{F}^{\alpha,q}_{p}(W)}^\clubsuit$.
From this and Theorem \ref{lusin}, we infer that
\eqref{wxg} holds true.
Thus, to complete the proof of Theorem \ref{glambda},
it remains to show that,
for any $\vec{f}\in[\cs'_{\infty}(\rn)]^m$,
\begin{equation}\label{wxgy}
\lf\|\vec f\r\|
_{\dot{F}^{\alpha,q}_{p}(W)}^\clubsuit\lesssim\lf\|\vec f\r\|
_{\dot{F}^{\alpha,q}_{p}(W)}.
\end{equation}
Let
$\{A_Q\}_{Q\in\mathcal{D}}$
be a sequence of reducing operators of order $p$ for $W$.
For any $\vec{f}\in[\cs'_{\infty}(\rn)]^m$, let
\begin{equation}\label{gq}
\lf\|\vec {f}\r\|
_{\dot{F}^{\alpha,q}_{p}(\{A_Q\})}^\clubsuit:=
\lf\|\lf\{\sum_{j\in\zz}
\sum_{Q\in\mathcal{D}_j}
2^{j\alpha q}2^{jn}
\sup_{z\in Q}\int_\rn
\frac{|A_Q(\varphi_j\ast\vec{f})(y )|^q}
{(1+2^j|z-y|)^{\lambda nq}}\,dy
\mathbf{1}_Q\r\}^{1/q}\r\|_{L^p(\rn)}.
\end{equation}
To prove \eqref{wxgy},  we first show that,
for any $\vec{f}\in[\cs_\infty'(\rn)]^m$,
\begin{equation}\label{gqw}
\lf\|\vec f\r\|
_{\dot{F}^{\alpha,q}_{p}(W)}^{\clubsuit}
\lesssim
\lf\|\vec f\r\|
_{\dot{F}^{\alpha,q}_{p}(\{A_Q\})}^\clubsuit.
\end{equation}
Indeed, for any given $p\in(0,\infty)$
and $q\in(0,\infty]$, and for any $x\in\rn$ and
$j\in\zz$,
let
$$
\gamma_j(x):=
\sum_{Q\in\mathcal{D}_j}\lf\|W^{1/p}(x)A_Q^{-1}\r\|
\mathbf{1}_Q(x),
$$

$$
h_j(x):=
2^{j\alpha }2^{jn/q}
\lf[\int_\rn
\frac{|W^{1/p}(x)(\varphi_j\ast\vec{f})(y)|^q}
{(1+2^j|x-y|)^{\lambda nq}}\,dy\r]^{1/q},
$$
and
$$
f_j(x):=
\sum_{Q\in\mathcal{D}_j}
\lf|Q\r|^{-\alpha/n}2^{jn/q}
\lf[\sup_{z\in Q}\int_\rn
\frac{|A_Q(\varphi_j\ast\vec{f})(y)|^q}
{(1+2^j|z-y|)^{\lambda nq}}\,dy\r]^{1/q}
\mathbf{1}_Q(x).
$$
It is obvious that,
for any $j\in\zz$ and $x\in\rn$,
\begin{align}\label{gkk}
h_j(x)
&=\sum_{Q\in\mathcal{D}_j}
2^{j\alpha }2^{jn/q}
\lf[\int_\rn
\frac{|W^{1/p}(x)A_Q^{-1} A_Q(\varphi_j\ast\vec{f})(y)|^q}
{(1+2^j|x-y|)^{\lambda nq}}\,dy\r]^{1/q}
\mathbf{1}_Q(x)\\
&\leq\sum_{Q\in\mathcal{D}_j}
\lf|Q\r|^{-\alpha/n}2^{jn/q}
\lf\|W^{1/p}(x)A_Q^{-1}\r\|
\lf[\int_\rn\frac{|A_Q(\varphi_j\ast\vec{f})(y)|^q}
{(1+2^j|x-y|)^{\lambda nq}}\,dy\r]^{1/q}
\mathbf{1}_Q(x)\noz\\
&\leq\sum_{Q\in\mathcal{D}_j}
\lf|Q\r|^{-\alpha/n}2^{jn/q}
\lf\|W^{1/p}(x)A_Q^{-1}\r\|
\lf[\sup_{z\in Q}\int_\rn
\frac{|A_Q(\varphi_j\ast\vec{f})(y)|^q}
{\lf(1+2^j|z-y|\r)^{\lambda nq}}\,dy\r]^{1/q}
\mathbf{1}_Q(x)\noz\\
&\leq\gamma_j(x)f_j(x).\noz
\end{align}
Notice that $f_j$ is a constant
on any given  $Q\in\mathcal{D}_j$,
which implies that $E_j(f_j)=f_j$,
where $E_j$ is the same as in  Lemma \ref{c38}.
By this, \eqref{gl*},  \eqref{gkk},
Lemma \ref{c38}, and \eqref{gq},
we conclude that,
for any $\vec{f}\in[\cs'_\infty(\rn)]^m$,
\begin{align*}
\lf\|\vec{f}\r\|
_{\dot{F}^{\alpha,q}_{p}(W)}^\clubsuit
&=\lf\|\lf\{h_j\r\}_{j\in\zz}\r\|_{L^p(\ell^q)}
\leq\lf\|\lf\{\gamma_j E_j(f_j)\r\}_{j\in\zz}\r\|
_{L^p(\ell^q)}\\	&\lesssim
\lf\|\lf\{E_j(f_j)\r\}_{j\in\zz}\r\|
_{L^p(\ell^q)}
\sim\lf\|\lf\{f_j\r\}_{j\in\zz}\r\|
_{L^p(\ell^q)}
\sim\lf\|\vec{f}\r\|
_{\dot{F}^{\alpha,q}_{p}(\{A_Q\})}^\clubsuit.
\end{align*}
Thus,
\eqref{gqw} holds true for any
$\vec{f}\in[\cs'_\infty(\rn)]^m$.

Next, we prove that, for any $\vec{f}\in[\cs_{\infty}'(\rn)]^m$,
\begin{equation}\label{pdq}
\lf\|\vec f\r\|
_{\dot{F}^{\alpha,q}_{p}(\{A_Q\})}^\clubsuit
\lesssim\lf\|\vec f\r\|
_{\dot{F}^{\alpha,q}_{p}(W)}.
\end{equation}
Let $\lambda\in(0,\infty)$ and
$j\in\zz$.
Then we claim that,
for any $z\in\rn$ and $A\in(0,q]$,
\begin{equation}\label{rd}
\sup_{v\in\rn}
\frac{|A_Q(\varphi_j\ast\vec{f})(v)|}
{(1+2^j|z-v|)^{\lambda n}}
\lesssim
\lf[2^{jn}\int_\rn
\frac{|A_Q(\varphi_j\ast\vec{f})(y)|^A}
{(1+2^j|z-y|)^{\lambda nA}}\,dy\r]^{1/A}.
\end{equation}
Then we prove \eqref{rd} by considering the following two cases on $A$.

Case 1) $A\in(0,1]$. In this case, using the assumption that $\varphi$
satisfies (T2) of Definition \ref{FBW},
we can then easily prove that there exists
a $\psi\in\cs(\rn)$ such
that $\supp\widehat{\psi}$ is bounded
away from the origin
and
$\widehat{\psi}=1$
on $\supp\widehat{\varphi}$.
By this and Lemma \ref{a.e.},
we find that
\begin{equation}\label{hde}
\varphi_j\ast\vec{f}=\psi_j\ast\varphi_j\ast\vec{f}
\end{equation}
 on $\rn$.
From this and the estimate that
\begin{equation}\label{guodushishi}
(1+2^{j}|z-v|)^{-1}
\leq(1+2^{j}|z-y|)^{-1}(1+2^j|v-y|)
\end{equation}
for any $j\in\zz$ and
$z,v,y\in\rn$,
we deduce that, for any given $A\in(0,\min\{1,q\}]$
and for any $z\in\rn$,
\begin{align}\label{guodu}
&\sup_{v\in\rn}
\frac{|A_Q(\varphi_j\ast\vec{f})(v)|}
{(1+2^j|z-v|)^{\lambda n}}\\
&\quad\leq\sup_{v\in\rn}
\frac{\int_\rn
|A_Q(\varphi_j\ast\vec{f})(y)2^{jn}
\psi(2^j(v-y))|\,dy}{(1+2^j|z-v|)
^{\lambda n}}\noz\\
&\quad\leq\sup_{v\in\rn}
\int_\rn
\frac{|A_Q(\varphi_j\ast\vec{f})(y)2^{jn}
\psi(2^j(v-y))|(1+2^j|v-y|)
^{\lambda n}}
{(1+2^j|z-y|)^{\lambda n}}\,dy\noz\\
&\quad\lesssim 2^{jn}\int_\rn
\frac{|A_Q(\varphi_j\ast\vec{f})(y)|}
{(1+2^j|z-y|)^{\lambda n}}\,dy\noz\\
&\quad\lesssim\sup_{v\in\rn}
\lf[\frac{|A_Q(\varphi_j\ast\vec{f})(v)|}
{(1+2^j|z-v|)
^{\lambda n}}\r]^{1-A}
\int_\rn 2^{jn}\frac{|A_Q(\varphi_j\ast\vec{f})(y)|^A}
{(1+2^j|z-y|)^{\lambda nA}}\,dy.\noz
\end{align}
When $a\in(n/\min\{1,p,q\}+\beta/p,\infty)$,
by $\vec{f}\in\dot{F}^{\alpha,q}_{p}(W)$,
\eqref{21},  and Lemma \ref{eq}, we find that
\begin{equation*}
\lf\{\sum_{Q\in\mathcal{D}_j}
2^{j\alpha}\sup_{v\in\rn}
\frac{|A_Q(\varphi_j\ast\vec{f})(v)|}
{(1+2^j|\cdot-v|)^a}
\mathbf{1}_Q\r\}_{j\in\zz}\in L^p(\ell^q),
\end{equation*}
which implies that $\sup_{v\in\rn}
\frac{|A_Q(\varphi_j\ast\vec{f})(v)|}
{(1+2^j|\cdot-v|)^{a}}<\infty
$ almost everywhere on $\rn$.
Using this, we find  that there exists a measurable
set $F\subset\rn$
satisfying that $|F|=0$ and,
for any $x\in\rn\setminus F$,
\begin{align}\label{cbd}
\sup_{v\in\rn}
\frac{|A_Q(\varphi_j\ast\vec{f})(v)|}
{(1+2^j|x-v|)^{a}}<\infty.
\end{align}
Then, for any $e\in F$, there exists an $x_e\in\rn\setminus F$
such that $2^j|x_e-e|<1/2$.
From this, we deduce that
\begin{align*}
\sup_{v\in\rn}\frac{|A_Q(\varphi_j\ast\vec{f})(v)|}
{(1+2^j|e-v|)^{a}}&\leq\sup_{v\in\rn}\frac
{|A_Q(\varphi_j\ast\vec{f})(v)|}
{(1+2^j|x_e-v|-2^j|x_e-e|)^{a}}\\&\leq
\sup_{v\in\rn}\frac{|A_Q(\varphi_j\ast\vec{f})(v)|}
{(1/2+2^j|x_e-v|)^{a}}<\infty,
\end{align*}
which, combined with \eqref{cbd}, implies that
$$\sup\limits_{v\in\rn}
\frac{|A_Q(\varphi_j\ast\vec{f})(v)|}
{(1+2^j|\cdot-v|)^{\lambda n}}<\infty
$$
on $\rn$ if $\lambda\in(\frac1{\min\{1,p,q\}}+\frac\beta{np},\infty)$.
By this and \eqref{guodu},
we conclude that \eqref{rd} holds true for $A\in(0,1]$.

Case 2) $A\in(1,\infty)$. In this case, from \eqref{hde},
 \eqref{guodushishi},
the H\"older inequality, and the change of variables,
we infer that, for any
given $A\in(1,q]$ and for any $z\in\rn$,
\begin{align*}
&\sup_{v\in\rn}\frac{|A_Q(\varphi_j\ast\vec{f})(v)}{(1+2^j|z-v|
)^{\lambda n}}\\
&\quad\leq\sup_{v\in\rn}
\frac{\int_\rn
|A_Q(\varphi_j\ast\vec{f})(y)2^{jn}
\psi(2^j(v-y))|\,dy}{(1+2^j|z-v|)
^{\lambda n}}\noz\\
&\quad\leq\sup_{v\in\rn}
\int_\rn
\frac{|A_Q(\varphi_j\ast\vec{f})(y)2^{jn}
\psi(2^j(v-y))|(1+2^j|v-y|)
^{\lambda n}}
{(1+2^j|z-y|)^{\lambda n}}\,dy\noz\\
&\quad\leq\lf[2^{jn}\int_{\rn}\frac{|A_Q(\varphi_j\ast\vec{f})(y)|)^A}
{(1+2^j|z-y|)^{\lambda nA}}\,dy\r]^{1/A}\noz\\
&\qquad\times\sup_{v\in\rn}\lf\{\lf[2^{jn}\int_{\rn}|\psi(2^j(v-y))|^{A'}
(1+2^j|v-y|)^{\lambda nA'}\,dy\r]^{1/A'}\r\}\noz\\
&\quad\sim\lf[2^{jn}\int_{\rn}\frac{|A_Q(\varphi_j\ast\vec{f}(y)|)^A}
{(1+2^j|z-y|)^{\lambda nA}}\,dy\r]^{1/A}\noz,
\end{align*}
which implies that \eqref{rd} holds true for $q\in(1,\infty)$.

Thus, \eqref{rd} holds true for any  $A\in(0,q]$.
Using \eqref{rd},
we obtain,
for any  given
$A\in(0,q]$
and any $z\in\rn$,
\begin{align}\label{jiang}
&2^{jn}\int_\rn
\frac{|A_Q(\varphi_j\ast\vec{f})(y)|^q}
{(1+2^j|z-y|)^{\lambda nq}}\,dy\\
&\quad\leq\sup_{v\in\rn}
\lf[\frac{|A_Q(\varphi_j\ast\vec{f})(v)|}
{(1+2^j|z-v|)^{\lambda n}}\r]
^{q-A}2^{jn}
\int_\rn\frac{|A_Q(\varphi_j\ast\vec{f})(y)|^A}
{(1+2^j|z-y|)^{\lambda nA}}\,dy\noz\\
&\quad\lesssim\lf[2^{jn}\int_\rn
\frac{|A_Q(\varphi_j\ast\vec{f})(y)|^A}
{(1+2^j|z-y|)^{\lambda nA}}\,dy\r]^{q/A}.\noz
\end{align}
Notice that,
for any $k,u\in\zz^n$,
$j\in\zz$,
$z\in Q_{ju}$,
and $y\in Q_{jk}$,
\begin{equation}\label{uxy}
(1+2^j|z-y|)
\sim 1+|k-u|.
\end{equation}
Since
$\lambda\in(1/\min\{1,p,q\}+\beta/(np),\infty)$,
then it follows that there exists an $A\in(0,\min\{1,p,q\})$
such that
$A[\lambda
-\frac{\beta}{np}]>1$.
By this, \eqref{uxy},  \eqref{jiang},
Lemma \ref{wsz},
the disjointness of $Q_{ju}$ for any $u\in\zz^n$,
and Lemma \ref{jcf},
we conclude that, for any $x\in\rn$,
\begin{align*}
&\sum_{j\in\zz}
\sum_{Q\in\mathcal{D}_j}
2^{j\alpha q}2^{jn}
\sup_{z\in Q}\int_\rn
\frac{|A_Q
(\varphi_j\ast\vec{f})(y)|^q}{(1+2^j|z-y|)^{\lambda nq}}\,dy
\mathbf{1}_Q(x)\\
&\quad\lesssim
\sum_{j\in\zz}\sum_{u\in\zz^n}
\sup_{z\in Q_{ju}}
\lf[2^{jn}\sum_{k\in\zz^n}
\int_{Q_{jk}}
\frac{\|A_{Q_{ju}}A_{Q_{jk}}^{-1}\|^A|
2^{j\alpha}A_{Q_{jk}}
(\varphi_j\ast\vec{f})(y)|^A}{(1+2^j|z-y|)
^{\lambda nA}}\,dy\r]^{q/A}
\mathbf{1}_{Q_{ju}}(x)\\
&\quad\lesssim
\sum_{j\in\zz}
\lf[\sum_{u\in\zz^n}
\sum_{k\in\zz^n}
\lf(1+\lf|k-u\r|\r)
^{-{\lambda nA}
+\frac{\beta A}{p}}2^{jn}
\int_{Q_{jk}}\lf|2^{j\alpha}A_{Q_{jk}}
\lf(\varphi_j\ast\vec{f}\r)(y)\r|^A\,dy
\mathbf{1}_{Q_{ju}}(x)\r]^{q/A}\\
&\quad\lesssim
\sum_{j\in\zz}\lf[\mathcal{M}
\lf(\sum_{Q\in \mathcal{D}_j}
\lf(\lf|2^{j\alpha}A_Q\lf(\varphi_j\ast\vec{f}\r)\r|
\mathbf{1}_Q\r)^A\r)(x)\r]^{q/A}.	
\end{align*}
From this, $A\in(0,\min\{p,q\})$, the Fefferman--Stein
vector-valued maximal inequality,
and Lemma \ref{eq},
we deduce that,
for any $\vec{f}\in[\cs_\infty'(\rn)]^m$,
\begin{align*}
\lf\|\vec f\r\|_{\dot{F}^{\alpha,q}_{p}(\{A_Q\})}^\clubsuit
&\lesssim\lf\|\lf\{\sum_{j\in\zz}
\lf[\mathcal{M}\lf(\sum_{Q\in \mathcal{D}_j}
\lf[\lf|2^{j\alpha}A_Q\varphi_j\ast\vec{f}\r|
\mathbf{1}_Q\r]^A\r)\r]^{q/A}\r\}^{1/q}\r\|_{L^p(\rn)}\\
&\sim\lf\|\lf[\mathcal{M}
\lf(\sum_{Q\in \mathcal{D}_j}
\lf[\lf|2^{j\alpha}A_Q\varphi_j\ast\vec{f}\r|
\mathbf{1}_Q\r]^A\r)\r]_{j\in\zz}\r\|
_{L^{p/A}(\ell^{q/A})}^{1/A}\\
&\lesssim\lf\|\vec f\r\|
_{\dot{F}^{\alpha,q}_{p}(\{A_Q\})}\sim\lf\|\vec f\r\|
_{\dot{F}^{\alpha,q}_{p}(W)},
\end{align*}
which implies that \eqref{pdq} holds true.
By both \eqref{gqw} and \eqref{pdq},  we obtain \eqref{wxgy}.
Using
\eqref{wxgy} and \eqref{wxg},  we find that \eqref{glambdad}
holds true for any $\vec{f}\in[\cs_{\infty}'(\rn)]^m$, which
completes the proof of Theorem \ref{glambda}.
\end{proof}
\begin{remark}\label{gr}
\begin{itemize}
\item [(\rm{i})] Theorem \ref{glambda} when  $m=1$, $W=1$
and $q\in(0,\infty)$
is a part of \cite[Theorem 3.2]{C} which is the
$g_{\lambda}^{*}$-function
characterization of Triebel--Lizorkin spaces
$\dot{F}_p^{\alpha,q}(\rn)$.
\item [(\rm{ii})]
Let
$p\in(0,\infty)$, $q\in(0,\infty)$,
$W\in A_p(\rn,\mathbb{C}^m)$,
and $\beta$ be the
doubling exponent of $W$.
We should point out that the range of $\lambda$ in Theorem \ref{glambda}
does not coincide with
the  one in \cite[Theorem 3.2]{C}, namely,
$(\frac{1}{\min\{p,q\}},\infty)$.
It is still unclear whether or not
Theorem \ref{glambda}
still holds true when $\lambda\in(\frac{1}{\min\{p,q\}},\frac{1}{\min\{1,p,q\}}
+\frac{\beta}{np}]$.
\end{itemize}
\end{remark}

\section{Fourier Multiplier\label{s6}}

In this section,
we study the mapping property on
$\dot{F}^{\alpha,q}_{p}(W)$
for a class of Fourier multipliers, which
was originally introduced by  Cho \cite{C}.

First, we denote by the \emph{symbol} $C(\rn\setminus\{\mathbf{0}\})$
the space of all continuous functions on $\rn\setminus\{\mathbf{0}\}$ and
recall the definition of the space $C^{\ell}(\rn\setminus\{\mathbf{0}\})$.
For any $\ell\in\nn$, let
$$C^{\ell}(\rn\setminus\{\mathbf{0}\}):=\lf\{f\in C(\rn
\setminus\{\mathbf{0}\}):\
\ \partial^{\sigma}f\in C(\rn\setminus\{\mathbf{0}\}),
\ \forall\,\sigma\in\zz^n\ \mathrm{and} \ |\sigma|\leq\ell\r\}.$$
For a given $\ell\in\mathbb{N}$ and a given $s\in\rr$,
assume that $m\in C^\ell(\rn\setminus\{\mathbf{0}\})$
satisfies that,
for any $\sigma\in\zz^n_+$
and $|\sigma|\leq\ell$,
\begin{equation}\label{Hormander}
\sup_{R\in(0,\infty)}
\lf[R^{-n+2s+2|\sigma|}
\int_{R\leq|\xi|<2R}
\lf|\partial^\sigma m(\xi)\r|^2\,d\xi\r]
\leq A_\sigma<\infty.
\end{equation}

\begin{remark}
When $s=0$ and $\ell\in\nn$,
\eqref{Hormander} is known as
the H\"ormander condition
(see, for instance, \cite[p. 263]{S2}).
Typical examples are given by the kernels of the
Riesz transforms $R^{(d)}$, where
$$\widehat{\left(R^{(d)}f\right)}(\xi):=-i(\xi_d/|\xi|)\widehat{f}(\xi)$$ for
any $\xi:=(\xi_1,\dots,\xi_n)\in\rn\setminus\{\mathbf{0}\}$,
$f\in\cs(\rn)$, and $d\in\{1,\,\dots,\,n\}$.
When $s\neq 0$ and $\ell\in\nn$,
a typical example satisfying \eqref{Hormander} is
given by
$m(\xi):=|\xi|^{-s}$ for any $\xi\in\rn\setminus\{\mathbf{0}\}$.
\end{remark}

Let $K$ be a compact set of $\rn$. Then
$\vec{f}:=(f_1,\,\dots,\,f_m)^T\in[\cs(\rn)]^m$
or
$\vec{f}\in[\cs'(\rn)]^m$
is said to have compact support set $K$,
denoted by
$\supp\vec{f}\subset K$,
if, for any $d\in\{1,\,\dots,\, m\}$,
$\supp f_d\subset K$.
From \cite[Theorem 2.3.21]{G1},
it follows that,
for any $\vec{f}\in[\cs'(\rn)]^m$
with $\supp\widehat{\vec{f}}\subset K$,
$\vec{f}\in[L^1_{\loc}(\rn)]^m$.
The following lemma
is a part of \cite[Corollary 6.13]{B33}.

\begin{lemma}\label{inftyp}
Let $K$ be a compact subset of $\rn$
and $W\in A_p(\rn,\mathbb{C}^m)$.
If $p\in(0,1)$ and $N\in(\beta/p+n,\infty)\cap\zz_+$,
or if $p\in(1,\infty)$
and $N\in(\beta/p,\infty)\cap\zz_+$, where $\beta$ is
the doubling exponent of $W$,
then there exists a positive constant $C$ such that,
for any $\vec{f}\in[\cs'(\rn)]^m$
with $\supp\widehat{\vec{f}}\subset K$,
\begin{equation}\label{ip}
\sup_{x\in\rn}
\frac{|\vec{f}(x)|}{(1+|x|)^N}
\leq C
\lf[\int_\rn\lf|\lf(W^{1/p}\vec{f}\r)(x)\r|^p\,dx\r]^{1/p}.
\end{equation}
Moreover, if
$\supp\widehat{\vec{f}}\subset\{x\in\rn:|x|\leq 2^j\}$,
where $j\in\zz_+$,
then, for any $N$ as above, there exists
a positive constant $C$ such that
\begin{equation}\label{cons}
\sup_{x\in\rn}
\frac{|\vec{f}(x)|}{(1+b|x|)^N}
\leq C4^{jn/p}
\lf[\int_\rn
\lf|\lf(W^{1/p}\vec{f}\r)(x)\r|^p\,dx\r]^{1/p}.
\end{equation}
\end{lemma}

\begin{lemma}\label{welldefinedk}
Let $\ell\in\mathbb{N}$ and $m\in C^\ell(\rn\setminus\{\mathbf{0}\})$
be the same as in  \eqref{Hormander} with $s\in\rr$. Then $m\in\cs'(\rn)$.
\end{lemma}

\begin{proof}
By $m\in C^\ell(\rn\setminus\{\mathbf{0}\})$, we conclude that,
for any $\varphi\in\cs(\rn)$,
\begin{align}\label{mvarphi}
\mathrm{I}:&=\int_{\rn} \lf|m(x)\varphi(x)\r|\,dx\\
&=\sum_{j\in\zz_+}\int_{2^j<|x|\leq 2^{j+1}}\lf| m(x)\varphi(x)\r|\,dx
+\int_{|x|\leq 1}\cdots\noz\\
&\lesssim\rho_{\mu,0}(\varphi)\lf[\sum_{j\in\zz_+}\int_{2^j<|x|
\leq 2^{j+1}}\frac{|m(x)|}{(1+|x|)^N}\,dx+\max\{|m(x)|:\,0<|x|\leq 1\}\r]
\noz\\
&\lesssim\rho_{\mu,0}(\varphi)\lf[\sum_{j\in\zz_+}\int_{2^j<|x|
\leq 2^{j+1}}\frac{|m(x)|}{2^{jN}}\,dx+\max\{|m(x)|:\,0<|x|\leq 1\}\r],\noz
\end{align}
where $\mu\in\zz_+^n$ with $|\mu|\leq N$, $N$ can be chosen
as any positive integer, and $\rho_{\mu,0}(\varphi)$ is the
same as in  \eqref{rhocs}.
Notice that, if $-n+2s\in[0,\infty)$, then $2^{j(-N+1)}\leq 2^{j(-n+2s)}$
for any $j,N\in\zz_+$; if $-n+2s\in(-\infty,0)$,
by choosing an $N\in(n-2s+1,\infty)$,
we then have $2^{j(-N+1)}\leq 2^{j(-n+2s)}$ for any $j,N\in\zz_+$.
From this, \eqref{mvarphi},  and \eqref{Hormander},
we deduce that, for any $\varphi\in\cs(\rn)$,
\begin{align*}
\mathrm{I}&\lesssim\rho_{N,0}(\varphi)\lf[\sum_{j\in\zz_+}
2^{-j}\int_{2^j<|x|\leq 2^{j+1}}|m(x)|2^{j(-N+1)}\,dx+
\max\{|m(x)|:\,0<|x|\leq 1\}\r]\\
&\lesssim\rho_{N,0}(\varphi)\lf[\sum_{j\in\zz_+}2^{-j}
2^{j(-n+2s)}\int_{2^j<|x|\leq 2^{j+1}}|m(x)|\,dx+
\max\{|m(x)|:\,0<|x|\leq 1\}\r]\\
&\lesssim\rho_{N,0}(\varphi)\lf[\sum_{j\in\zz_+}2^{-j}
A_0+\max\{|m(x)|:\,0<|x|\leq 1\}\r]
\sim\rho_{N,0}(\varphi),
\end{align*}
which implies that $m\in\cs'(\rn)$. This finishes the
proof of Lemma \ref{welldefinedk}.
\end{proof}
Let $\ell\in\nn$ and $m\in C^{\ell}(\rn\setminus\{\mathbf{0}\})$
be the same as in  \eqref{Hormander} with  $s\in\rr$. By
Lemma \ref{welldefinedk}, we can define the \emph{Fourier multiplier}
$T_m$ by setting, for any $\vec{f}\in[\cs_{\infty}(\rn)]^m$,
\begin{equation}\label{multiplieryuan}
\lf(T_{m}\vec{f}\r)^{\land}:=
m\widehat{\vec{f}}.
\end{equation}
Furthermore, let $K$ be the distribution whose Fourier transform is $m$.

Then we show that,
via a suitable way,
$T_{m}$ can be defined on the space
$\dot{F}^{\alpha,q}_{p}(W)$.
To this end,
let $\varphi,\psi\in\cs(\rn)$
satisfy both (T2) of Definition \ref{FBW} and \eqref{cf}.
For any $\vec{f}\in\dot{F}^{\alpha,q}_{p}(W)$
and $\phi\in\cs_\infty(\rn)$,
let
\begin{equation}\label{multiplier}
\lf\langle T_{m}\vec{f},\phi\r\rangle:=
\sum_{j\in\zz}\vec{f}
\ast\varphi_j
\ast\psi_j\ast\widetilde{\phi}\ast K(\mathbf{0}),
\end{equation}
where $\widetilde{\phi}=\phi(-\cdot)$.
It is obvious that, when $\vec{f}\in[\cs_{\infty}(\rn)]^m$,
then both $T_m\vec{f}$ in \eqref{multiplieryuan} and
$T_m\vec{f}$ in \eqref{multiplier} coincide in $[\cs_{\infty}'(\rn)]^m$.
The following result shows that
the right-hand side of \eqref{multiplier}
converges and
$T_{m}\vec{f}$ in \eqref{multiplier}
is well defined for any $\vec{f}\in\dot{F}^{\alpha,q}_{p}(W)$.

\begin{lemma}\label{defined}
Let  $\alpha\in\rr$,
$p\in(0,\infty)$,
$q\in(0,\infty]$,
$W\in A_p(\rn,\mathbb{C}^m)$,
and $\vec{f}\in\dot{F}^{\alpha,q}_{p}(W)$.
Let $\ell\in(\beta/p+n/\min\{1,p\}+\frac{n}{2},\infty)$,
where $\beta$ is the doubling exponent of $W$, and let
$m\in C^{\ell}(\rn\setminus\{\mathbf{0}\})$
be the same as in  \eqref{Hormander} with $s\in\rr$.
Then $T_{m}$ in \eqref{multiplier}
is independent of the choice of the pair
$(\varphi,\psi)$ of Schwartz functions
satisfying both (T2) of Definition \ref{FBW} and \eqref{cf}.
Moreover,
$T_{m}\vec{f}\in[\cs_\infty'(\rn)]^m$ and $T_{m}\vec{f}$
in \eqref{multiplier} is well defined.
\end{lemma}

To show Lemma \ref{defined},
we need the following lemmas,
which are just \cite[Lemma 4.1(i)]{C}
and \cite[Lemma 2.2]{YY1}, respectively.

\begin{lemma}\label{kgj}
Let $\psi\in\cs(\rn)$
satisfy that $\widehat\psi$
has compact support away from the origin.
Let $\lambda\in(0,\infty)$,
$\ell\in(\lambda+n/2,\infty)$, and
$m$ be the same as in  \eqref{Hormander} with $s\in\rr$.
Let $K$ be the distribution whose Fourier transform is $m$.
Then there exists a positive constant $C$
such that, for any $j\in\zz$,
\begin{equation*}
\int_\rn\lf(1+2^j|z|\r)^\lambda
\lf|\lf(K\ast\psi_j\r)(z)\r|\,dz
\leq C2^{-js}.
\end{equation*}
\end{lemma}

\begin{lemma}\label{varphigj}
For any given $M\in\mathbb{N}$,
there exists a positive constant $C$
such that,
for any
$\varphi,\psi\in\cs_\infty(\rn)$,
$j,\ell\in\zz$, and $x\in\rn$,
$$
\lf|\varphi_j\ast\psi_\ell(x)\r|
\leq C \lf\|\varphi\r\|_{S_{M+1}}
\lf\|\psi\r\|_{S_{M+1}}
2^{-|\ell-j|M}
\frac{2^{-\min\{j,\ell\}M}}
{(2^{-\min\{j,\ell\}}+|x|)^{n+M}},
$$
where, for any $\varphi\in\cs_{\infty}(\rn)$,
$$
\lf\|\varphi\r\|_{S_{M}}:=
\sup_{\{\gamma\in\zz_+^n:\,|\gamma|\leq M\}}
\sup_{x\in\rn}\lf|\partial^\gamma
\varphi(x)\r|\lf(1+|x|\r)^{n+M+|\gamma|}.
$$
\end{lemma}
\begin{proof}[Proof of Lemma \ref{defined}]
Let $\varphi$ and $\psi$ be a pair of Schwartz
functions satisfying both (T2) of Definition \ref{FBW} and
\eqref{cf}. Let $\varphi^\star$
and $\psi^\star$ be another
pair of Schwartz functions
satisfying both (T2) of Definition \ref{FBW} and \eqref{cf}.
By this, Lemma \ref{fz}, and $\phi\in\cs_\infty(\rn)$,
we find that
\begin{equation}\label{zsp}
\widetilde{\phi}=
\sum_{t\in\zz}
\varphi_t^\star
\ast\psi_t^\star
\ast\widetilde{\phi}
\quad \text{in}\quad \cs_\infty(\rn).
\end{equation}
Since $\varphi$ and $\varphi^\star$ satisfy (T2)
of Definition \ref{FBW},
it follows that
there exists an
$L\in\mathbb{N}$ such that,
for any $|j-t|> L$,
\begin{equation}\label{0}
\varphi_j\ast\varphi^\star_{t}=0.
\end{equation}
Let $\alpha,p,q,W$, and $m$ be the same as in
this lemma. Let $\vec{f}\in\dot{F}^{\alpha,q}_{p}(W)$.
To prove $$\sum_{j\in\zz}\vec{f}\ast\varphi_j\ast\psi_j
\ast\widetilde{\phi}\ast K(\mathbf{0})$$ converges,
where $K$ is the distribution whose Fourier transform is $m$,
we first show that $$\sum_{j\in\zz}\sum_{t\in\zz}\vec{f}\ast\varphi_j
\ast\psi_j\ast\varphi_t^\star\ast\psi_t^\star\ast\widetilde{\phi}\ast
 K(\mathbf{0})$$
converges. To this end,
by both \eqref{cons} of Lemma \ref{inftyp} and Definition \ref{FBW},
we find that, for any $N\in\zz_+$ satisfying Lemma \ref{inftyp} and for any $\vec{f}\in\dot{F}^{\alpha,q}_{p}(W)$ and $j\in\zz_+$,
\begin{align}\label{yd}
\sup_{x\in\rn}\frac{|\varphi_j\ast\vec{f}(x)|}{(1+2^j|x|)^N}&\lesssim
2^{2jn/p}\lf\{\int_{\rn}\lf|W^{1/p}(x)\lf(\varphi_j\ast\vec{f}\r)(x)
\r|^p\,dx\r\}^{1/p}\\
&\lesssim 2^{j(n/p-\alpha)}\lf\|\vec f\r\|_{\dot{F}^{\alpha,q}_{p}(W)}\noz.
\end{align}
From this, $\varphi^\star_j\ast\psi_j^{\star}=(\varphi^\star\ast\psi^{\star})_j
\in\cs_{\infty}(\rn)$, Lemma \ref{varphigj}, and the estimate that
$$1+2^j|z|\leq\left(1+2^j|y|\right)\left(1+2^j|z-y|\right)$$
for any $j\in\zz$ and $y,z\in\rn$, and
Lemma \ref{kgj} with $\ell>N+n/2$, we infer that
\begin{align*}
&\sum_{j=0}^{\infty}\lf|\vec{f}\ast\varphi_j\ast\psi_j\ast\varphi_j
^\star\ast\psi_j^\star\ast\widetilde{\phi}\ast K(\mathbf{0})\r|\\
	&\quad\leq\sum_{j=0}^{\infty}\int_\rn\lf|\vec{f}\ast\varphi_j(-z)
\r|\lf|\lf(\varphi^\star\ast\psi^\star\r)_j\ast\widetilde{\phi}
\ast\psi_j\ast K(z)\r|\,dz\\
	&\quad\lesssim\sum_{j=0}^\infty2^{j(2n/p-\alpha)}
\lf\|\vec f\r\|_{\dot{F}^{\alpha,q}_{p}(W)}
\int_{\rn}\int_\rn\lf(1+2^j|z|\r)^N\\
&\qquad\times\lf|\lf(\varphi^{\star}\ast\psi^{\star}\r)_j\ast\widetilde{\phi}(z-y)
\r|\lf|\psi_j\ast K(z)\r|\,dy\,dz\noz\\
	&\quad\lesssim\sum_{j=0}^\infty2^{j(2n/p-\alpha-M)}\lf\|\vec f
\r\|_{\dot{F}^{\alpha,q}_{p}(W)}\int_\rn\int_\rn\frac{(1+2^j|z|)^N}
{(1+|z-y|)^{n+M}}\lf|\psi_j\ast K(y)\r|\,dy\,dz\noz\\
	&\quad\lesssim\sum_{j=0}^\infty2^{j(N+2n/p-\alpha-M)}
\lf\|\vec f\r\|_{\dot{F}^{\alpha,q}_{p}(W)}\int_\rn\int_\rn\frac{(1+2^j|y|)^N}
{(1+|z-y|)^{n+M-N}}\lf|\psi_j\ast K(y)\r|\,dy\,dz\noz\\
	&\quad\sim\sum_{j=0}^\infty2^{j(N+2n/p-\alpha-M)}
\lf\|\vec f\r\|_{\dot{F}^{\alpha,q}_{p}(W)}\int_\rn
\lf(1+2^j|y|\r)^{N}\lf|\psi_j\ast K(y)\r|\,dy\noz\\
	&\quad\lesssim\sum_{j=0}^\infty2^{j(N+2n/p-\alpha-M-s)}
\lf\|\vec f\r\|_{\dot{F}^{\alpha,q}_{p}(W)}
	\sim\lf\|\vec f\r\|_{\dot{F}^{\alpha,q}_{p}(W)},\noz
	\end{align*}
where the implicit positive constants depend on
$\psi^{\star},\varphi^{\star}$, and $\phi$,
and where $M\in\mathbb{N}$ is chosen to be sufficiently
large such that $M>\max\{N, N+2n/p-s-\alpha\}$.
On the other hand, by \eqref{yd},  Lemma \ref{varphigj},
and the estimate that $1+2^j|z|\leq(1+2^j|y|)(1+2^j|z-y|)$
for any $j\in\zz$ and $y,z\in\rn$, and Lemma \ref{kgj},
we find that
\begin{align*}
	&\sum_{j=-\infty}^{-1}\lf|\vec{f}\ast\varphi_j\ast
\psi_j\ast\varphi_j^\star\ast\psi_j^\star\ast\widetilde
{\phi}\ast K(\mathbf{0})\r|\\
	&\quad\leq\sum_{j=-\infty}^{-1}\int_\rn\lf|\vec{f}
\ast\varphi_j(-z)\r|\lf|\lf(\varphi^\star\ast\psi^\star\r)_j
\ast\widetilde{\phi}\ast\psi_j\ast K(z)\r|\,dz\noz\\
	&\quad\lesssim\sum_{j=-\infty}^{-1}2^{j(-N-\alpha)}\lf\|
\vec f\r\|_{\dot{F}^{\alpha,q}_{p}(W)}\int_{\rn}\int_\rn
\lf(1+2^j|z|\r)^N
\noz\\
	&\quad\quad\times\lf|\lf(\varphi^{\star}\ast\psi^{\star}\r)_j
\ast\widetilde{\phi}(z-y)\r|\lf|\psi_j\ast K(z)\r|\,dy\,dz\noz\\
	&\quad\lesssim\sum_{j=-\infty}^{-1}2^{j(-N-\alpha+n+M)}
\lf\|\vec f\r\|_{\dot{F}^{\alpha,q}_{p}(W)}\int_\rn\int_\rn
\frac{(1+2^j|z|)^N}{(1+2^j|z-y|)^{n+M}}
\noz\\
	&\quad\quad\times\lf|\psi_j\ast K(y)\r|\,dy\,dz\noz\\
	&\quad\lesssim\sum_{j=-\infty}^{-1}2^{j(-N-\alpha+n+M)}
\lf\|\vec f\r\|_{\dot{F}^{\alpha,q}_{p}(W)}\int_\rn\int_\rn
\frac{(1+2^j|y|)^N}{(1+2^j|z-y|)^{n+M-N}}\noz\\
	&\quad\quad\times\lf|\psi_j\ast K(y)\r|\,dy\,dz\noz\\
	&\quad\sim\sum_{j=-\infty}^{-1}2^{j(-N-\alpha+M)}
\lf\|\vec f\r\|_{\dot{F}^{\alpha,q}_{p}(W)}\int_\rn
\lf(1+2^j|y|\r)^{N}\lf|\psi_j\ast K(y)\r|\,dy\noz\\
	&\quad\lesssim\sum_{j=-\infty}^{-1}2^{j(-N-\alpha+M-s)}
\lf\|\vec f\r\|_{\dot{F}^{\alpha,q}_{p}(W)}
	\sim\lf\|\vec f\r\|_{\dot{F}^{\alpha,q}_{p}(W)},\noz
	\end{align*}
	where the implicit positive constants depend on
$\psi^{\star},\varphi^{\star}$, and $\phi$,
and where $M\in\mathbb{N}$ is chosen to be sufficiently
large such that $M>\max\{N, s+\alpha+N\}$.
From an argument similar to that used in
the above two estimations, we deduce that
\begin{equation}\label{sphi}
\sum_{j\in\zz}\sum_{t=j-L}^{j+L}\lf|\vec{f}\ast\varphi_j
\ast\psi_j\ast\varphi_t^\star\ast\psi_t^\star\ast
\widetilde{\phi}\ast K(\mathbf{0})\r|<\infty.
\end{equation}
By \eqref{sphi} and \eqref{0} with $|j-t|>L$, we conclude that
\begin{align*}
&\sum_{j\in\zz}\sum_{t\in\zz} \lf|\vec{f}\ast\varphi_j
\ast\psi_j\ast\varphi_t^\star\ast\psi_t^\star\ast\widetilde{\phi}
\ast K(\mathbf{0})\r|\\
&\quad=\sum_{j\in\zz}\sum_{t=j-L}^{j+L}\lf|\vec{f}
\ast\varphi_j\ast\psi_j\ast\varphi_t^\star\ast\psi_t^\star
\ast\widetilde{\phi}\ast K(\mathbf{0})\r|<\infty.
\end{align*}
Using \eqref{ip} of Lemma \ref{inftyp}, Definition \ref{FBW},
Lemma \ref{varphigj}, the estimate that
$$1+2^{t}|x|\leq\left(1+2^{t}|y|\right)\left(1+2^{t}|x-y|\right)$$
for any $x,y\in\rn$ and $t\in\zz$, and Lemma \ref{kgj},
we find that, for any fixed $j\in\zz$,
\begin{align*}
&\sum_{t\in\zz}\int_{\rn}\lf|\vec{f}\ast\varphi_j\ast\psi_j(-x)
\lf(\varphi_t^{\star}\ast\psi^{\star}_{t}\ast\widetilde{\phi}
\ast K\r)(x)\r|\,dx\\
&\quad\lesssim\lf\|\vec f\r\|_{\dot{F}^{\alpha,q}_{p}(W)}
\sum_{t\in\zz}\int_{\rn}(1+|x|)^{N}\int_{\rn}\lf|\psi^{\star}_{t}
\ast\widetilde{\phi}(x-y)\r|\lf|\varphi_{t}^{\star}\ast K(y)\r|
\,dy\,dx\noz\\
&\quad\lesssim\lf\|\vec f\r\|_{\dot{F}^{\alpha,q}_{p}(W)}
\lf[\sum_{t=0}^{\infty}\int_{\rn}\int_{\rn}\frac{2^{-t R}
(1+|x|)^{N}}{(1+|x-y|)^{n+R}}\lf|\varphi_{t}^{\star}\ast K(y)\r|
\,dy\,dx\r.\noz\\
&\qquad\lf.+\sum_{t=-\infty}^{-1}\int_{\rn}\int_{\rn}
\frac{2^{t(n+R)}(1+|x|)^N}{(1+2^{t}|x-y|)^{n+R}}\lf|
\varphi_{t}^{\star}\ast K(y)\r|\,dy\,dx\r]\noz\\
&\quad\lesssim\lf\|\vec f\r\|_{\dot{F}^{\alpha,q}_{p}(W)}
\lf[\sum_{t=0}^{\infty}\int_{\rn}\int_{\rn}\frac{2^{-t R}
(1+|x|)^{N}}{(1+|x-y|)^{n+R}}\lf|\varphi_{t}^{\star}\ast
K(y)\r|\,dy\,dx\noz\r.\\
&\qquad\lf.+\sum_{t=-\infty}^{-1}\int_{\rn}\int_{\rn}
\frac{2^{t(n+R-N)}(1+2^{t}|x|)^{N}}{(1+2^{t}|x-y|)^{n+R}}
\lf|\varphi_{t}^{\star}\ast K(y)\r|\,dy\,dx\r]\noz\\
&\quad\lesssim\lf\|\vec f\r\|_{\dot{F}^{\alpha,q}_{p}(W)}
\lf[\sum_{t=0}^{\infty}\int_{\rn}\int_{\rn}\frac{2^{-t R}
(1+|y|)^N}{(1+|x-y|)^{n+R-N}}\lf|\varphi_{t}^{\star}\ast K(y)
\r|\,dy\,dx\noz\r.\\
&\lf.\qquad+\sum_{t=-\infty}^{-1}2^{t(n+R-N)}\int_{\rn}\int_{\rn}
\frac{(1+2^{t}|y|)^N}{(1+2^{t}|x-y|)^{n+R-N}}\lf|\varphi_{t}^{\star}
\ast K(y)\r|\,dy\,dx\r]\noz\\
&\quad\lesssim\lf\|\vec f\r\|_{\dot{F}^{\alpha,q}_{p}(W)}
\lf[\sum_{t=0}^{\infty}\int_{\rn}\int_{\rn}\frac{2^{-t R}(1+2^{t}|y|)^N}
{(1+|x-y|)^{n+R-N}}\lf|\varphi_{t}^{\star}\ast K(y)\r|\,dy\,dx\r.\noz\\
&\qquad+\lf.\sum_{t=-\infty}^{-1}2^{t(n+R-N)}\int_{\rn}\int_{\rn}
\frac{(1+2^{\ell}|y|)^N}{(1+2^{t}|x-y|)^{n+R-N}}\lf|\varphi_{t}^{\star}
\ast K(y)\r|\,dy\,dx\r]\noz\\
&\quad\sim\lf\|\vec f\r\|_{\dot{F}^{\alpha,q}_{p}(W)}\lf[\sum_{t=0}^{
\infty}2^{-t R}\int_{\rn}\lf(1+2^{t}|y|\r)^N\lf|\varphi_{t}^{\star}\ast
K(y)\r|\,dy\r.\noz\\
&\qquad\lf.+\sum_{t=-\infty}^{-1}2^{t(R-N)}\int_{\rn}\lf(1+2^{t}|y|\r)^N
\lf|\varphi_{t}^{\star}\ast K(y)\r|\,dy\r]\noz\\
&\quad\sim\lf\|\vec f\r\|_{\dot{F}^{\alpha,q}_{p}(W)}\lf[\sum_{t=0}^{
\infty}2^{-t R}\int_{\rn}\lf(1+2^{t}|y|\r)^N\lf|\varphi_{t}^{\star}\ast
K(y)\r|\,dy\noz\r.\\
&\qquad+\lf.\sum_{t=-\infty}^{-1}2^{t(R-N)}\int_{\rn}\lf(1+2^{t}|y|\r)^N\lf|
\varphi_{t}^{\star}\ast K(y)\r|\,dy\r]\noz\\
&\quad\lesssim\lf\|\vec f\r\|_{\dot{F}^{\alpha,q}_{p}(W)}\lf[\sum_{t=0}^{
\infty}2^{-t (R+s)}+\sum_{t=-\infty}^{-1}2^{t(R-N-s)}\r]
\sim\lf\|\vec f\r\|_{\dot{F}^{\alpha,q}_{p}(W)},\noz
\end{align*}
where the implicit positive constants depend on $j,\alpha,p$,
and $n$, and where $R\in\mathbb{N}$ is chosen to be sufficiently large such
that $R>\max\{N,\,-s,\,N+s\}$ with $N$ the same as in  \eqref{yd}.
From this and \eqref{zsp},  we deduce that
\begin{align*}
&\vec{f}\ast\varphi_j\ast\psi_j\ast\widetilde{\phi}\ast K(\mathbf{0})\\
&\quad=\int_{\rn}\vec{f}\ast\varphi_j\ast\psi_j(-x)\widetilde{\phi}\ast
K(x)\,dx\\
&\quad=\int_{\rn}\vec{f}\ast\varphi_j\ast\psi_j(-x)\lf\langle K,
\widetilde{\phi}(x-\cdot)\r\rangle\,dx\\
&\quad=\int_{\rn}\vec{f}\ast\varphi_j\ast\psi_j(-x)\sum_{t\in\zz}\lf\langle
K, \varphi_{t}^{\star}\ast\psi_{t}^{\star}\ast\widetilde{\phi}(x-\cdot)
\r\rangle\,dx\\
&\quad=\sum_{t\in\zz}\int_{\rn}\vec{f}\ast\varphi_j\ast\psi_j(-x)\lf\langle
K, \varphi_{t}^{\star}\ast\psi_{t}^{\star}\ast\widetilde{\phi}(x-\cdot)
\r\rangle\,dx\\
&\quad=\sum_{t\in\zz}\vec{f}\ast\varphi_j\ast\psi_j\ast\varphi_{t}^{\star}
\ast\psi_{t}^{\star}\ast\widetilde{\phi}\ast K(\mathbf{0}).
\end{align*}
By this and \eqref{0},  we find that
\begin{align*}
\sum_{j\in\zz}\vec{f}\ast\varphi_j\ast\psi_j\ast\widetilde{\phi}\ast
K(\mathbf{0})&=\sum_{j\in\zz}\sum_{t\in\zz}\vec{f}\ast\varphi_j\ast\psi_j
\ast\varphi_t^\star\ast\psi_t^\star\ast\widetilde{\phi}\ast K(\mathbf{0}),\noz
\end{align*}
which, together with \eqref{zsp},  $K\in\cs'(\rn)\subset\cs_{\infty}'(\rn)$,
and Lemma \ref{fz}, further implies that
\begin{align*}
	&\sum_{j\in\zz}\vec{f}\ast\varphi_j\ast\psi_j\ast\widetilde{\phi}\ast
K(\mathbf{0})\\
	&\quad=\sum_{t\in\zz}\sum_{j\in\zz}\vec{f}\ast\varphi_t^\star\ast
\psi_t^\star\ast\varphi_j\ast\psi_j\ast\widetilde{\phi}\ast K(\mathbf{0})\\
	&\quad=\sum_{t\in\zz}\sum_{j\in\zz}\int_{\rn}\vec{f}\ast\varphi_{t}^{\star}
\ast\psi_{t}^{\star}(-x)\lf\langle K, \varphi_j\ast\psi_j\ast\widetilde{\phi}
(\cdot-x)\r\rangle\,dx\\
	&\quad=\sum_{t\in\zz}\int_{\rn}\vec{f}\ast\varphi_{t}^{\star}\ast\psi_{t}^{
\star}(-x)\sum_{j\in\zz}\lf\langle K, \varphi_j\ast\psi_j\ast\widetilde{
\phi}(\cdot-x)\r\rangle\,dx\\
	&\quad=\sum_{t\in\zz}\int_{\rn}\vec{f}\ast\varphi_{t}^{\star}\ast\psi_{t}^{
\star}(-x)\lf\langle K,\widetilde{\phi}(\cdot-x)\r\rangle\,dx\\
	&\quad=\sum_{t\in\zz}\vec{f}\ast\varphi_t^\star\ast\psi_t^\star\ast
\widetilde{\phi}\ast K(\mathbf{0}).
\end{align*}
Thus, $T_{m}\vec{f}$ in \eqref{multiplier} is independent of the choice of the pair
$(\varphi,\psi)$ satisfying both (T2) of Definition \ref{FBW} and \eqref{cf}. Moreover,
the above argument also implies that $T_{m}\vec{f}\in[\cs'_\infty(\rn)]^m$ is well
defined, which completes the proof of Lemma \ref{defined}.
\end{proof}
\begin{lemma}\label{t*}
 Let $p\in(0,\infty)$ and $W\in A_p(\rn,\mathbb{C}^m)$. Assume that  $\varphi,
 \psi\in\cs(\rn)$ satisfy both (T1) and (T2) of Definition \ref{FBW} and \eqref{cf}.
 Let $\lambda\in(\beta/p+n/\min\{1,p\},\infty)$, where $\beta$ is the doubling exponent of $W$, and let $\ell\in(\lambda+n/2,\infty)$
and $m\in C^{\ell}(\rn\setminus\{
 \mathbf{0}\})$ be the same as in  \eqref{Hormander} with $s\in\rr$. Let $\phi:=
 \varphi\ast\psi$ and $T_m$ be the same as in  \eqref{multiplier}.
 Then there exists a positive constant $C$ such that, for any $\vec{f}\in
 \dot{F}^{\alpha,q}_{p}(W)$ and $x\in\rn$,
$$\lf|W^{1/p}(x)\lf(T_{m}\vec{f}\ast\phi_j\r)(x)\r|\leq C 2^{-js}\sup_{y\in\rn
}\frac{|W^{1/p}(x)(\varphi_j\ast\vec{f})(y)|}{(1+2^j|x-y|)^\lambda}.$$
\end{lemma}
\begin{proof}
Let all the symbols be the same as in  the present lemma.
Let $K$ be the distribution whose Fourier transform is $m$, $\phi=\varphi\ast\psi
$, and $\vec{f}\in\dot{F}^{\alpha,q}_{p}(W)$.
We first show that, for any $j\in\zz$, $$\vec{f}\ast\varphi_j\ast\psi_j\ast
K\in[\cs_{\infty}'(\rn)]^m.$$
Indeed, by \eqref{yd},  the estimate that $1+2^j|x-y|\leq(1+2^j|x|)(1+2^j|y|)$ for
any $x,y\in\rn$, and  Lemma \ref{kgj}, we conclude that, for any $
\gamma\in\cs_{\infty}(\rn)$,
\begin{align*}
&\lf|\lf\langle \vec{f}\ast\varphi_j\ast\psi_j\ast K,\gamma\r\rangle\r|\\
&\quad
\leq\int_{\rn}\int_{\rn}\lf|\vec{f}\ast\varphi_j(x-y)\r|\lf|\psi_j
\ast K(y)\r|\lf|
\gamma(x)\r|\,dy\,dx\\
&\quad\lesssim\lf\|\vec{f}\r\|_{\dot{F}^{\alpha,q}_{p}(W)}\int_{\rn}\int_{\rn}
\lf(1+2^j|x-y|\r)^N\lf|\psi_j\ast K(y)\r|\lf|\gamma(x)\r|\,dy\,dx\\
&\quad\lesssim\lf\|\vec{f}\r\|_{\dot{F}^{\alpha,q}_{p}(W)}\int_{\rn}\int_{\rn}
\lf(1+2^j|x|\r)^N\lf(1+2^j|y|\r)^N\lf|\psi_j\ast K(y)\r|\lf|\gamma(x)\r|
\,dy\,dx\\
&\quad\lesssim\lf\|\vec{f}\r\|_{\dot{F}^{\alpha,q}_{p}(W)}\int_{\rn}\lf(1+2^j|x|\r)^N
\lf|\gamma(x)\r|\,dx\\
&\quad\lesssim\lf\|\vec{f}\r\|_{\dot{F}^{\alpha,q}_{p}(W)}\rho_{N+2n,0}(\gamma),
\end{align*}
where the implicit positive constants depend on $\alpha\in\rr$, $j\in\zz$,
and $s\in\rr$, and where $N$ is the same as in  \eqref{yd} and $\rho_{N+2n,0}$ is
the same as in  \eqref{rhocs},  which implies that the above claim holds true. Using
this claim and Lemma \ref{fz}, we find that
$$\sum_{\ell\in\zz}\vec{f}\ast\varphi_j\ast\psi_j\ast K\ast\varphi_{\ell}
\ast\psi_\ell$$
converges in $[\cs_{\infty}'(\rn)]^m$.
From \eqref{multiplier} and Lemmas \ref{defined} and \ref{fz},
we deduce that,
for any
$\vec{f}\in\dot{F}^{\alpha,q}_{p}(W)$,
$T_{m}\vec{f}\in[\cs'_\infty(\rn)]^m$ and, for any
$\eta\in\cs_\infty(\rn)$,
\begin{align}\label{56x}
\langle T_m\vec{f}\ast\phi_j,\eta\rangle&=\langle T_m\vec{f},
\eta\ast\widetilde{\phi_j}\rangle\\
&=\sum_{\ell\in\zz}\vec{f}\ast\varphi_\ell\ast\psi_{\ell}\ast
\widetilde{\eta}\ast\varphi_j\ast\psi_j\ast K(\mathbf{0})\noz\\
&=\sum_{\ell\in\zz}\lf\langle\vec{f}\ast\varphi_{\ell}\ast\psi_{\ell}
\ast\varphi_{j}\ast\psi_{j}\ast K,\eta\r\rangle\noz\\
&=\lf\langle\sum_{\ell\in\zz}\vec{f}\ast\varphi_{\ell}\ast\psi_{\ell}
\ast\varphi_{j}\ast\psi_{j}\ast K,\eta\r\rangle\noz\\
&=\lf\langle\vec{f}\ast\varphi_j\ast\psi_j\ast K,\eta\r\rangle.\noz
\end{align}
Let $\gamma\in\cs(\rn)$ satisfy both $\widehat{\gamma}=1$ on $\supp{
\widehat{\phi}}$ and $\supp{\gamma}\subseteq\{x\in\rn:0<|x|<\pi\}$.
For any $j\in\zz$ and $x\in\rn$, let $\gamma_j(x):=2^{jn}\gamma(2^jx)$.
Then, by \eqref{56x} with $\gamma\in\cs_{\infty}(\rn)$,
\cite[Theorem 2.3.21]{G1},
and  $\phi=\varphi\ast\psi$ with $\supp\widehat{\phi}\subset
\{x\in\rn:|x|<\pi\}$,
 we find that
\begin{align*}
T_m\vec{f}\ast\phi_j(x)&=T_m\vec{f}\ast\phi_j\ast
 \gamma_j(x)
= \vec{f}\ast\varphi_j\ast\psi_j\ast K\ast\gamma_j(x)\\ &=\vec{f}\ast
\varphi_j\ast\psi_j\ast K(x)
\end{align*}
for any $x\in\rn.$
From this and Lemma \ref{kgj},
we infer that, for any $x\in\rn$,
\begin{align*}
&\lf|W^{1/p}(x)
\lf(T_{m}\vec{f}\ast\phi_j\r)(x)\r|\\
&\quad\leq\int_\rn
\frac{|W^{1/p}(x)(\varphi_j\ast\vec{f})(x-y)|}
{(1+2^j|y|)^\lambda}\lf(1+2^j|y|\r)^\lambda
\lf|\lf(K\ast\psi_j\r)(y)\r|\,dy\\
&\quad\lesssim 2^{-js}\sup_{y\in\rn}
\frac{|W^{1/p}(x)(\varphi_j\ast\vec{f})(y)|}
{(1+2^j|x-y|)^\lambda}.
\end{align*}
This finishes the proof of Lemma \ref{t*}.
\end{proof}

\begin{theorem}\label{fourier multiplier}
Let $\alpha\in\rr$,
$p\in(0,\infty)$,
$q\in(0,\infty]$, and
$W\in A_p(\rn,\mathbb{C}^m)$.
Let $
\ell\in(\frac{n}{\min\{1,p,q\}}
+\frac{\beta}{p}+\frac{n}{2},\infty)
$ and $m\in C^{\ell}(\rn\setminus\{\mathbf{0}\})$ be the same as in
\eqref{Hormander}
with $s\in\rr$, where $\beta$ is the doubling exponent of $W$.
Let $T_m$ be the same as in
\eqref{multiplier}.
Then there exists a
positive constant $C$ such that, for any $\vec{f}\in\dot{F}^{\alpha,q}_{p}(W)$,
$$
\lf\|T_{m}\vec{f}\r\|
_{\dot{F}^{\alpha+s,q}_{p}(W)}
\leq C\lf\|\vec f\r\|
_{\dot{F}^{\alpha,q}_{p}(W)}.
$$
\end{theorem}
\begin{proof}
Let $\phi$, $\varphi$, and $\psi$ be the same as in the present lemma.
Since $\phi=\varphi\ast\psi$, it follows that $\phi$ satisfies both (T1)
and (T2)
of Definition \ref{FBW}. Using this, Definition \ref{FBW}, Lemmas \ref{eq}
and \ref{t*}, and Theorem \ref{pmf}, we find that
	\begin{align*}
	\lf\|T_{m}\vec{f}\r\|_{\dot{F}^{\alpha+s,q}_{p}(W)}
	&=\lf\|\lf[\sum_{j\in\zz}\lf|2^{j(\alpha+s)}W^{1/p}\lf(T_{m}\vec{f}
\ast\phi_j\r)\r|^q\r]^{1/q}\r\|_{L^p(\rn)}\\
	&\lesssim\lf\|\lf[\sum_{j\in\zz}\lf|2^{j\alpha }\sup_{y\in\rn}
\frac{|W^{1/p}(\cdot)(\varphi_j\ast\vec{f})(y)|}{(1+2^j|\cdot-y|)^\lambda}
\r|^q\r]^{1/q}\r\|_{L^p(\rn)}\\
	&\sim\lf\|\vec f\r\|_{\dot{F}^{\alpha,q}_{p}(W)},
	\end{align*}
	which completes the proof of Theorem \ref{fourier multiplier}.
\end{proof}
\begin{remark}\label{tr}
Theorem \ref{fourier multiplier} when $m=1$
and $W=1$ is a part of \cite[Theorem 1.5(i)]{YYZ}.
\end{remark}

\smallskip

\noindent\textbf{Acknowledgements}\quad The authors would like to thank
Fan Bu and Professor Wen Yuan for some helpful
discussions on the subject of this article.



\bigskip

\noindent Qi Wang, Dachun Yang (Corresponding author), and Yangyang Zhang

\smallskip

\noindent  Laboratory of Mathematics and Complex Systems
(Ministry of Education of China),
School of Mathematical Sciences, Beijing Normal University,
Beijing 100875, The People's Republic of China

\smallskip

\noindent{{\it E-mails}:}
\texttt{wqmath@mail.bnu.edu.cn} (Q. Wang)

\noindent\phantom{{\it E-mails}:}
\texttt{dcyang@bnu.edu.cn} (D. Yang)

\noindent\phantom{{\it E-mails}:}
\texttt{yangyzhang@mail.bnu.edu.cn} (Y. Zhang)

\end{document}